\documentclass[reqno,11pt]{amsart}
\usepackage{amscd,amssymb,verbatim}
\usepackage{enumitem}
\usepackage{marginnote}
\usepackage[nobysame]{amsrefs}
\usepackage{hyperref}
\usepackage[scr=rsfso]{mathalfa}
\usepackage{xcolor}
\usepackage{slashed}
\usepackage{graphicx}
\usepackage{subcaption}
\usepackage{mathtools}

\numberwithin{equation}{section}

\theoremstyle{plain}
\newtheorem{theorem}{Theorem}[section]
\newtheorem{lemma}[theorem]{Lemma}
\newtheorem{corollary}[theorem]{Corollary}

\newtheorem{conjecture}[theorem]{Conjecture}
\newtheorem{assumption}[theorem]{Assumption}

\theoremstyle{definition}
\newtheorem{definition}[theorem]{Definition}

\theoremstyle{remark}
\newtheorem{remark}[theorem]{Remark}

\newcommand{\RR}{\mathbb{R}}
\newcommand{\CC}{\mathbb{C}}

\newcommand{\ZZ}{\mathbb{Z}}
\renewcommand{\>}{\rangle}
\newcommand{\<}{\langle}

\newcommand{\ind}{\operatorname{ind}}
\newcommand{\supp}{\operatorname{supp}}
\newcommand{\End}{\operatorname{End}}

\newcommand{\p}{\partial}
\newcommand{\dirac}{\slashed{\p}}

\newcommand{\B}{\mathcal{B}}

\newcommand{\calF}{\mathscr{F}}

\newcommand{\calE}{\mathscr{E}}

\newcommand{\calH}{\mathscr{H}}

\newcommand{\ID}{\operatorname{Id}}
\renewcommand{\span}{\operatorname{span}}

\newcommand{\n}{\nabla}

\begin{document}

%\hfill{Preliminary version: \today}
\title[$\ZZ_2$-valued index of elliptic odd symmetric operators]{On the $\ZZ_2$-valued index of elliptic odd symmetric operators on non-compact manifolds}
\author{Maxim Braverman and Ahmad Reza Haj Saeedi Sadegh }
%thanks{Supported in part by the NSF grant DMS-1005888.}
\address{Department of Mathematics,
        Northeastern University,
        Boston, MA 02115,
        USA
         }


\begin{abstract}
We investigate elliptic operators with a symmetry that forces their index to vanish. We study the secondary index, defined modulo 2. We examine Callias-type operators with this symmetry on non-compact manifolds and establish mod 2 versions of the Gromov-Lawson relative index theorem, the Callias index theorem, and the  Boutet de Monvel's index theorem for Toeplitz operators.

\medskip

\noindent\textsc{R\'esum\'e.}
Nous étudions des opérateurs elliptiques avec une symétrie qui force leur indice à s'annuler.
Nous étudions l'indice secondaire, défini modulo 2. Nous examinons les opérateurs de type Callias avec
cette symétrie sur des variétés non compactes et établissons des versions mod 2 du théorème de l'indice relatif de Gromov-Lawson, du théorème de l'indice de Callias, et du théorème de l'indice de Boutet de Monvel pour les opérateurs de Toeplitz.
\end{abstract}
%%%%%%%%%%%%%%%%%%%%%%%

\maketitle

%------------------------------------------------------
%------------------------------------------------------
\section{Introduction}\label{S:introduction}

In many cases, when symmetry causes the index of a Fredholm operator to vanish, a secondary invariant taking values in $\ZZ_2$ can be defined. The first example of this was the $\ZZ_2$-valued index of a skew-adjoint operator on a real Hilbert space, discovered by Atiyah and Singer almost 50 years ago, \cite{AtSinger69}. Since then, many other examples of symmetries have been considered, drawing on applications in physics, cf. \cites{Schulz-Baldes15,DollSB21} and references therein (see also \cites{GrafPorta13,Hayashi17}). 

This paper focuses on the case of \textit{odd symmetric operators}, studied in \cites{Schulz-Baldes15,DeNittisSB15,DollSB21} and relevant to physical systems with time-reversing symmetries. From a mathematical perspective, this translates to the existence of an anti-unitary anti-involution, i.e., an anti-linear map  $\tau$ with $\tau^2=-1$. A Fredholm operator $D$ is called \textit{odd symmetric} if $\tau D \tau^{-1} = D^*$.

The $\ZZ_2$-valued index (or $\tau$-index) of $D$ is defined as
\[\ind_\tau D \ := \ \dim\ker D \qquad \text{modulo} \quad 2.\]
This definition resembles the Atiyah-Singer index for skew-adjoint operators, but here $D$ is not skew-adjoint. The reason for this definition is that the $\tau$-index is invariant under continuous homotopy (\cites{DeNittisSB15,DollSB21}) and shares many other properties with the usual $\ZZ$-valued index. Moreover, it is shown in  \cites{Schulz-Baldes15}, that the space $\mathcal{F}_{\tau}(H)$ of odd symmetric Fredholm operators on an infinite-dimensional separable Hilbert space, $H$, is the classifying space for the $KO^{-2}$ functor. In particular, $\pi_0(\mathcal{F}_{\tau}(H))=KO^{-2}(\textit{pt})\simeq\mathbb{Z}_2$. Therefore operators $A,B\in \mathcal{F}_{\tau}(H)$ belong to the same connected component of $\mathcal{F}_{\tau}(H)$ if and only if they have the same $\tau$-index (see Remark~\ref{R:KO-2} for more details). 

This paper investigates the $\tau$-index of first-order elliptic odd symmetric differential operators on both compact and non-compact manifolds. We obtain analogs of several results for the usual index in the odd symmetric case. The proofs often closely resemble those in the classical case, although some modifications are necessary. Firstly, to ensure that all constructions respect the $\tau$-action, and secondly, the notion of a topological index and the Atiyah-Singer-type index theorem for the $\tau$-index is currently absent.

One of the applications of the $\tau$-index is in computing the bulk index in the theory of topological insulators with time-reversal symmetry, particularly for 2D topological insulators of type AII \cites{GrafPorta13, Hayashi17}. Graf and Porta \cite{GrafPorta13} established a bulk-edge correspondence for these insulators, comparing the $\tau$-valued index with a $\ZZ_2$-valued spectral flow. We address this problem in \cite{BrSaeedi24a} and establish a generalization of the Graf-Porta result to a broader class of operators. This is similar to how the usual $\ZZ$-valued bulk-edge correspondence is generalized in \cite{Br19Toeplitz}.

\medskip
Let us now provide a more detailed description of our main results and constructions.

Let $M$ be a Riemannian manifold and let $\theta:M\to M$ be a metric preserving involution. The pair $(M,\theta)$ is referred to as an involutive manifold. A \textit{quaternionic bundle} over $M$  is a complex vector bundle $E$ together with an anti-unitary anti-involution, which covers $\theta$,  \cites{Dupont69,DeNittisGomi15,Hayashi17}. By this, we mean that for every $x\in M$ we are given a map 
$\theta^E_x:E_x\to E_{\theta x}$, such that $\theta^E_{\theta x}\circ \theta^E_x=-1$. We define an anti-unitary anti-involution $\tau:\Gamma(M,E)\to \Gamma(M,E)$ by 
\[
\tau:\, f(x)\ \mapsto \ \theta^E f(\theta x), \qquad f\in\Gamma(M,E),
\]
and consider an elliptic first-order operator $D$ which is odd symmetric with respect to $\tau$. Suppose that $D$ is Fredholm. This is, for example, the case when $M$ is compact. But we are also interested in the case when $M$ is complete, but $D$ is still Fredholm, e.g. when $D$ is of Callias type, cf. \cites{Anghel93Callias,Anghel93,Bunke95}).

A quaternionic Dirac bundle above in the case of the trivial involution $\theta:M\to M$ gives the setting for the so-called $\textbf{Cl}_2$-Dirac operator (\cites{LawMic89}) which was originally studied by Atiyah and Singer (\cite{AtSinger5}). See Remark \ref{R:cl2dirac}

In Section~\ref{S:odddirac}, we study the basic properties of the $\tau$-index of an odd symmetric elliptic operator on a quaternionic vector bundle. We also compute the $\tau$-index in several geometric examples.  One of those examples corresponds to the bulk index of 2d Topological Insulator of type AII, cf. \cites{GrafPorta13,Hayashi17}.

In Section~\ref{S:relative index}, we consider the case when $M$ is a complete non-compact manifold and the operator is Callias-type (another name for it is a Dirac-Schr\"odinger operator, \cite{BruningMoscovici}). This means that we consider an operator $B= D+F$, where $D$ is a first-order elliptic operator with a uniformly bounded principal symbol, and $F$ is a bundle map, chosen so that $B$ is Fredholm (cf. Section~\ref{SS:Calliasoperators}).  We establish an analog of the Gromov-Lawson relative index theorem,\cite{GromovLawson83}.  

In Section~\ref{S:callias}, we specify to the case when $\dirac$ is an odd symmetric generalized Dirac operator on a complete involutive \textit{odd-dimensional} manifold $M$. Here, we introduce a Callias-type operator $B= \dirac+i\Phi$ and establish a $\mathbb{Z}_2$-valued variant of the Callias index theorem,  \cites{Anghel93Callias,Bunke95}. This theorem expresses the $\tau$-index of $B$ as the $\tau$-index of a particular generalized Dirac operator on a hypersurface. Although our proof is similar to \cite{BrCecchini17}, we needed to adapt several steps to maintain $\tau$-equivariance in our constructions. Moreover, we found several significant simplifications, which can also be applied to simplify the proof in the classical case. 

In Section~\ref{S:cobordism}, we apply the Callias index theorem to show that cobordisms preserve the $\ZZ_2$-valued index. 

In Section~\ref{S:Boutet de Monvel}, we prove a $\ZZ_2$-valued analog of the generalized Boutet de Monvel's index theorem for Toeplitz operators, \cite{BoutetdeMonvel78}.    Boutet de Monvel established his result for a Toeplitz operator on a pseudo-convex domain. Guentner and Higson, \cite{GuentnerHigson96}, reinterpreted it in terms of the index of a Callias-type operator. Bunke, \cite{Bunke00}, generalized their result to Callias-type operators on general manifolds. In this section, we present a $\ZZ_2$-valued analog of this generalized Boutet de Monvel theorem of Bunke. 

\subsection*{Acknowledgements} We would like to thank the referee for the very careful reading of the manuscript and for making valuable remarks and corrections.

%------------------------------------------------------
\section{The $\ZZ_2$-index of an odd symmetric operator}\label{S:tauindex}

In this section, we consider a (possibly) unbounded Fredholm operator $D$ which is symmetric with respect to a given anti-unitary anti-involution $\tau$. Such operators are called \textit{odd symmetric},  \cites{Schulz-Baldes15,DeNittisSB15,DollSB21}. The index of such an operator vanishes due to its symmetry. Following  \cites{DeNittisSB15,DollSB21}, we defined the {\em secondary invariant} –- the $\ZZ_2$-index of $D$ by $\ind_\tau D:= \dim\ker D$ modulo 2. This is similar to the Atiyah-Singer construction of the index of a skew-adjoint operator (though $D$ is not skew-adjoint). We show that the  $\tau$-index is invariant under homotopy and study other properties of this index. We also compute it in an example. Most of the material in this section is taken from  \cite{Schulz-Baldes15}, but, since we work with unbounded operators, some minor changes are necessary.

%---------------------------------------------
\subsection{An anti-unitary anti-involution}\label{SS:involtiontau}
Let $H$ be a complex Hilbert space and let $\tau:H\to H$ be an \textbf{anti-linear} bounded map satisfying $\tau^*\tau=1$.  Here  we denote by $\tau^*$ the unique anti–linear operator satisfying 
\[
	\<\tau x, y\>_H\ = \ \overline{\<x,\tau^*y\>_H} \qquad\text{for all}\quad
	x,y\in H.
\]
We also assume $\tau^*=-\tau$ so that $\tau^2=-1$.  We refer to such $\tau$ as an {\em anti-unitary anti-involution}.

\begin{lemma}\label{L:tau orthogonal}
Suppose $\tau:H\to H$ is an anti-unitary anti-involution. Then $\tau e$ is perpendicular to $e$ for every $e\in H$. Further, if $L\subset H$ is a $\tau$-invariant subspace and $e\in H$ is a vector perpendicular to $L$, then $\tau e$  is also perpendicular to $L$. 
 \end{lemma}
 \begin{proof}
 Since $\tau^*=-\tau$, we have
 \[
    \<\tau e,e\>_H \ = \ \overline{\<e,\tau^*e\>_H}  \ = \ -\,\<\tau e,e\>_H. 
 \]
 Hence, $\<\tau e, e\>_H=0$.

Suppose $L$ is a $\tau$-invariant subspace and let $e\in H$ be a vector perpendicular to $L$. For any $x\in L$,
 \[
    \<\tau e,x\>_H \ = \ \overline{\<e,\tau^*x\>_H}  \ = \ -\,\<\tau x,e\>_H
    \ = \ 0.
 \]
 \end{proof}

The following lemma is a version of the \textit{Kramers' degeneracy}, cf. \cite{KleinMartin1952} (see also \cite[Lemma~5.1]{DeNittisSB15}).

%-------------------------------------
\begin{lemma}\label{L:even} An anti-unitary anti-involution $\tau$ does not have non-zero fixed vectors, i.e. $\tau x=x$ implies $x=0$. Further, if $\tau$ acts on a finite-dimensional space $H$, then $\dim H$ is even.  Moreover, there exists a subspace $L\in H$ such that $\tau L$ is perpendicular to $L$ and $H$ decomposes into an orthogonal direct sum $H=L\oplus \tau L$.
\end{lemma}

\begin{proof}
If $\tau x= x$, then $-x= \tau^2 x= x$, hence, $x=0$.

Suppose now that $\dim H<\infty$. We will construct an orthogonal basis of $H$ by induction. Choose any non-zero vector $e_1\in H$. Let $e_2=\tau e_1$. By Lemma~\ref{L:tau orthogonal}, $e_1\perp e_2$.

Suppose we already found $2k$ linearly independent vectors $\{e_1,\ldots,e_{2k}\}$ in $H$, such that $e_{2j}= \tau e_{2j-1}$ for all $j=1,\ldots,k$. Then $\span\{e_1,\ldots,e_{2k}\}$ is a $\tau$-invariant space. Choose any non-zero vector $e_{2k+1}$ perpendicular to  $\span\{e_1,\ldots,e_{2k}\}$ and set $e_{2k+2}:= \tau e_{2k+1}$. By Lemma~\ref{L:tau orthogonal}, $e_{2k+2}$ is perpendicular to both $e_{2k+1}$ and $\span\{e_1,\ldots,e_{2k}\}$. Hence vectors $\span\{e_1,\ldots,e_{2k+2}\}$ are mutually orthogonal. 

By induction, we obtain an orthogonal basis of $H$ $\{e_1,\ldots,e_{2n}\}$ such that $e_{2j+2}:= \tau e_{2j+1}$, which has an even number of vectors. Hence, $\dim H$ is even. Set $L:= \span\{e_1,e_3,\ldots, e_{2n-1}\}$. Then $\tau L= \span\{e_2,e_4,\ldots, e_{2n}\}$ and $H=L\oplus \tau L$.
\end{proof}

%----------------------------------
\subsection{Odd symmetric operators}\label{SS:oddsymmetric}
Fix an anti-unitary anti-involution $\tau:H\to H$. 

\begin{definition}\label{D:odd symmetric}
A closed linear operator $D:H\to H$  is called {\em odd symmetric} (cf.  \cites{DeNittisSB15,DollSB21}) if 
\begin{equation}\label{E:tauDtau}
    \tau\, D\, \tau^{-1}\ = \ D^*.
\end{equation}
\end{definition}

Let $W\subset H$ be a dense subspace and let $D:H\to H$ be an unbounded closed linear operator on $H$ whose domain is $W$. Then $W$ is a Hilbert space with the inner product
\[
	\<x,y\>_{W}\ = \ \<x,y\>_H\ + \ \<Dx,Dy\>_{W}
\]
and the operator $D:W\to H$ is a bounded. We denote by $\calF_\tau(W,H)$ the space of Fredholm operators on $H$ with domain $W$, i.e.,  odd symmetric operators with closed range and finite-dimensional kernel and cokernel. The space $\calF_\tau(W,H)$ is endowed with the norm operator topology.

\begin{remark}
In our geometric applications, $W$ will be a version of the first Sobolev space $W^{1,2}$. Hence the notation.   
\end{remark}

%----------------------------------
\subsection{Graded odd symmetric operators}\label{SS:graded odd symmetric}
Assume now that $H=H^+\oplus H^-$ is a graded Hilbert space and that $D$ is \textit{odd with respect to this grading}, i.e. $D:H^\pm\to H^\mp$. Let $W^\pm:= W\cap H^\pm$. Then $W= W^+\oplus W^-$. We denote $D^\pm$ the restriction of $D$ to $H^\pm$ and write $D= \begin{psmallmatrix}
0&D^-\\D^+&0\end{psmallmatrix}$. We are mostly interested in the case when $D$ is self-adjoint. Then $(D^\pm)^*= D^\mp$.

Assume that the anti-involution $\tau$ is also odd with respect to the grading,  $\tau:H^\pm\to H^\mp$. If $D$ is odd symmetric, then \eqref{E:tauDtau} implies that $\tau D^\pm\tau^{-1}= D^\mp$. In this situation, we say that $D$ is a \textit{graded odd symmetric operator}.

We denote the space of graded Fredholm odd symmetric self-adjoint operators by $\widehat{\calF}_\tau(W,H)$.
If $D\in \widehat{\calF}_\tau(W,H)$, then $\tau D^\pm \tau^{-1}= D^\mp= (D^\pm)^*$. Then we say that \textit{the operators $D^\pm$ are odd symmetric}. Notice, that in this case $\tau:W\to W$.

%----------------------------------
\subsection{The $\ZZ_2$-valued index}\label{SS:Z2index}
Let $D$ be a Fredholm odd symmetric operator. It follows from \eqref{E:tauDtau} that $\ind D=0$. Following \cite{Schulz-Baldes15}, we define the {\em secondary} $\ZZ_2$-valued invariant – the \textit{$\tau$-index of $D$} by 
\begin{equation}\label{E:tauindex}
	\ind_\tau D\ := \ \dim \ker D \qquad \mod 2.
\end{equation}
This is similar to the Atiyah-Singer index of skew-adjoint Fredholm operators, \cite{AtSinger69}. 

If $D\in \widehat{\calF}_\tau(W,H)$ is a graded Fredholm odd symmetric self-adjoint operator on a graded Hilbert space $H$, we define 
\begin{equation}\label{E:tauindex graded}
	\ind_\tau D^+\ := \ \dim \ker D^+ \qquad \mod 2.
\end{equation}

The reason for these definitions is that the $\tau$-index is stable under homotopy of $D$:

%-------------
\begin{theorem}\label{T:homotopyindex}
\begin{enumerate}
\item The $\tau$-index $\ind_\tau D$ is constant on the connected components of\/ $\calF_\tau(W,H)$.

\item The $\tau$-index  $\ind_\tau D^+$ is constant on the connected components of $\widehat{\calF}_\tau(W,H)$.
\end{enumerate}
\end{theorem}

\begin{proof}
We prove the graded case. The proof in the ungraded case is practically the same. 

Let $D_s$ ($s\in[0,1]$) be a continuous path in $\widehat{\calF}_\tau(W,H)$. It suffices to show that for every $s_0\in[0,1]$ there exists $\epsilon>0$ such that $\ind_\tau D_s^+$ is independent of $s$ modulo 2 for $s\in (s_0-\epsilon,s_0+\epsilon)$.

We have  
\[
    \ind_\tau D_s^+\ =\ \dim \ker D_s^+\ = \ \dim\ker D_s^-D_s^+.
\]
For $s_0\in[0,1]$, there exists $\epsilon, \delta>0$ such that $\delta$ is not in the spectrum of $D_s^-D_s^+$ for all $s\in (s_0-\epsilon,s_0+\epsilon)$ and the intersection of the spectrum of $D_{s_0}^-D_{s_0}^+$ with the interval $[0,\delta]$ is equal to $\{0\}$. 

For $\lambda\in \RR$, let  $E_{s,\lambda}\subset H^+$ denote the eigenspace of $D_s^-D_s^+$ with eigenvalue $\lambda$. Our assumptions on $\epsilon$ and $\delta$ imply that, for all $s\in (s_0-\epsilon,s_0+\epsilon)$
\[
    \sum_{0\le\lambda<\delta} \dim E_{s,\lambda} \ = \  \dim \ker D_{s_0}^-D_{s_0}^+.
\]
To prove the homotopy invariance of $\ind_\tau D_s^+$ it now suffices to show that the dimension of $E_{s,\lambda}$ is even for all $0<\lambda<\delta$.

Using \eqref{E:tauDtau} and  $\tau^{-1}=-\tau$, we conclude that 
\[
    D_s^-D_s^+\,\tau D_s^+\ = \ \tau D_s^+\,  D_s^-D_s^+.
\]
It follows that $\tau D_s^+(E_{s,\lambda})\subset E_{s,\lambda}$.  For $\lambda>0$  consider the operator 
\[
    \calE_{s,\lambda}:= \frac1{\sqrt{\lambda}}\tau D_s^+:\, E_{s,\lambda}\ \to E_{s,\lambda}.
\]
Then 
\[
    \calE_{s,\lambda}^2 \ = \  \frac1\lambda\, \tau D_s^+ \tau D_s^+
    \ = \ - \frac1\lambda\tau D_s^+ \tau^{-1} D_s^+
    \ = \ - \frac1\lambda D_s^- D_s^+ \ = \ -\ID:\, E_{s,\lambda}\ \to E_{s,\lambda}.
\]
Thus $\calE_{s,\lambda}:\, E_{s,\lambda}\ \to E_{s,\lambda}$ is an anti-unitary anti-involution. From Kramers' degeneracy Lemma~\ref{L:even} we conclude now that the dimension of $E_{s,\lambda}$ is even for all $\lambda\in(0,\delta)$. 
\end{proof} 

\begin{corollary}\label{C:compact perturbation}
Let $D\in\calF_\tau(W,H)$ and let $K:W\to H$ be a compact odd symmetric operator. Then the operator $D+K$ and $\ind_\tau (D+K)= \ind_\tau D$.
\end{corollary}

\begin{remark}\label{R:KO-2}
    In \cites{Schulz-Baldes15}, it is shown that the space $\mathcal{F}_{\tau}(H)$ of odd symmetric Fredholm operators on an infinite-dimensional separable Hilbert space, $H$, is the classifying space for the $KO^{-2}$ functor; that is for a compact Hausdorff space $X$, there is a bijection
    \begin{equation}\label{eq:oddsymmko-2}
    KO^{-2}(X)\simeq [X,\mathcal{F}_{\tau}(H)].    
    \end{equation}
    In particular, $\pi_0(\mathcal{F}_{\tau}(H))=KO^{-2}(\textit{pt})\simeq\mathbb{Z}_2$. Therefore, by Theorem~\ref{T:homotopyindex}, the $\tau$-index defines an isomorphism $\ind_\tau:KO^{-2}(\textit{pt})\overset{\sim}{\longrightarrow}\mathbb{Z}_2$. In particular, operators $A,B\in \mathcal{F}_{\tau}(H)$ belong to the same connected component of $\mathcal{F}_{\tau}(H)$ if and only if they have the same $\tau$-index. Below, we briefly recall how the correspondence \eqref{eq:oddsymmko-2} is achieved in \cites{Schulz-Baldes15}.

    Let $J:H\to H$ denote the multiplication by the scalar $\sqrt{-1}$ and by $\langle\cdot,\cdot\rangle:H\times H\to \mathbb{C}$ its inner product. Using Lemma \ref{L:even}, one can show that there is a `complex conjugation' operator $C:H\to H$ and a `real' unitary operator $I:H\to H$ such that $I^2=-1$ and $\tau=CI$ (by real we mean $CIC=I$). One can see that the correspondence $T\mapsto IT$ gives a bijection from $\mathcal{F}_{\tau}$ to the space of antisymmetric Fredholm operators $\mathcal{F}_{as}(H)=\{T\in \mathcal{F}(H):T^t=-T\}$, where $T^t:=CT^*C$ is the transpose of $T$. Now the correspondence $B\mapsto BC$ gives a bijection from $\mathcal{F}_{as}(H)$ to $\mathcal{F}^2(H_{\mathbb{R}})$ where $H_{\mathbb{R}}$ is the real Hilbert given by the inner product $(\cdot,\cdot):=\textit{Re}\langle\cdot,\cdot\rangle$ on $H$, and $\mathcal{F}^2(H_{\mathbb{R}})$ is the space of real-linear skew-adjoint operators $A:H_{\mathbb{R}}\to H_{\mathbb{R}}$ with the property $AJ=-JA$, cf. \cites{AtSinger69}. Now the space $\mathcal{F}^2({H_{\mathbb{R}}})$ is a Classifying space for the $KO^{-2}$ functor as shown in \cites{AtSinger69}.
\end{remark}
%-----------------
\subsection{Example: A first-order differential operator on a line}\label{SS:DAline}
Let
\[
	\theta_\CC: \CC^2\ \to \ \CC^2, \qquad \theta_\CC:\,(z_1,z_2)\ \mapsto \ (\bar{z}_2,-\bar{z}_1).
\]
Clearly, $\theta$ is anti-linear, and $\theta_\CC= - \theta_\CC^* = -\theta_\CC^{-1}$. 

Let $H:=L^2(\RR,\CC^2)$ be the space of square-integrable functions with values in $\CC^2$ and let $W:= W^{1,2}(\RR,H)$, where $W^{1,2}$ denotes the Sobolev space of square-integrable functions, whose first derivative is also square-integrable. We write functions in $H$ as $f(t)= \big(f_1(t),f_2(t)\big)$
We define an anti-unitary anti-involution $\tau_\CC:H\to H$ by 
\begin{equation}\label{E:tauC2}
    \tau:\, f(t) \ \mapsto \ \theta_\CC f(-t).
\end{equation}

Consider the operator $D:W\to H$ defined by the formula
\begin{equation}\label{E:DC2}
    D\ : = \ \frac{d}{dt} \ + \ \begin{pmatrix}
		\arctan(t)&0\\0&-\arctan(t)
	\end{pmatrix}.
\end{equation}
One readily sees that $D$ is Fredholm and $\tau$-symmetric.
%-------------------------------------
\begin{lemma}\label{L:indexnormalization}
$\ind_\tau D= 1.$
\end{lemma}

%-------------------------------------
\begin{proof}
If $f (t)= \big(f _1(t),f _2(t)\big)$ is in the kernel of $D$, then 
\[
	\dot{f}_1(t)\ = \ -\arctan(t)\, f_1(t),\qquad
	\dot{f}_2(t)\ = \ \arctan(t) \,f_2(t).
\]
Hence, 
\[
	f_1(t) = \ C\,e^{-\int_0^t\arctan(s)ds},\qquad
	f_2(t) = \ C\,e^{\int_0^t\arctan(s)ds}.
\]
Clearly, $f_1(t)$ is square-integrable, while $f_2(t)$ is not square-integrable unless $C=0$. Thus, the kernel of $D$ consists of vectors $C\big(e^{-\int_0^t\arctan(s)ds},0\big)$ and, hence, has dimension one. 
\end{proof}

%------------------------------------------------------
%------------------------------------------------------
\section{The $\ZZ_2$-valued index of an odd symmetric elliptic operator} \label{S:odddirac}

In this section, we study the $\ZZ_2$-valued  index of odd symmetric elliptic operators on a manifold with involution and we compute it in several geometric examples. 

%---------------------------------------------
\subsection{Quaternionic vector bundles}\label{SS:quaternionic}
Let $M$ be a Riemannian manifold of dimension $n$.   Usually, we will assume that $M$ is either closed or complete.  Throughout the paper, we identify the tangent and the cotangent bundles to $M$ using the Riemannian metric and write $TM$ for both bundles. 

Let $\theta:M\to M$ be a metric preserving involution, $\theta^2=1$.  The pair $(M,\theta)$ is called an {\em involutive Riemannian manifold}.

\begin{definition}\label{D:quaternionic bundle}
A {\em quaternionic vector bundle} over an involutive manifold $M$ is a hermitian vector bundle $E$ endowed with an anti-linear anti-unitary bundle map $\theta^E:E\to \theta^*E$, such that $(\theta^E)^2=-1$. This means that for every $x\in M$ there exists an anti-linear map $\theta^E_x:E_x\to E_{\theta(x)}$, which depends smoothly on $x$ and satisfies $\theta^E_{\theta(x)}\, \theta^E_x=-1$.

If $E=E^+\oplus E^-$ is a graded vector bundle, we say that $\theta^E$ is {\em odd (with respect to the grading)} if $\theta^E(E^\pm)=E^\mp$. In this situation, we say that $(E,\theta^E)$ is a {\em graded quaternionic bundle}. 
\end{definition}

Given a quaternionic vector bundle $(E,\theta^E)$ over an involutive manifold  $(M,\theta)$ we define an anti-unitary anti-involution $\tau$ on the space $\Gamma(M,E)$ of sections of $E$ by 
\begin{equation}\label{E:tau=}
    \tau:\, f(x)\ \mapsto \ \theta^E f(\theta x), \qquad f\in\Gamma(M,E).  
\end{equation}

\begin{remark}\label{R:cl2dirac}
    When $M$ is of dimension $8k+2$ compact spin manifold with the involution $\theta:M\to M$ the identity map, the complex spinor bundle is a quaternionic bundle. In this setting, the usual index of the Dirac operator vanishes as it is odd symmetric, and its associated $\tau$-index equals the complex dimension of the space of the harmonic spinors (see \cites{AtSinger5} and \cites{LawMic89}).

\end{remark}

%---------------------------------------------
\subsection{Odd symmetric elliptic operators}\label{SS:oddelliptic}
Let $(E=E^+\oplus E^-,\theta^E)$ be a graded quaternionic vector bundle over an involutive manifold $(M,\theta)$. Let $D:\Gamma(M,E)\to \Gamma(M,E)$ be a self-adjoint odd symmetric elliptic differential operator on $M$, which is odd with respect to the grading, i.e. $D:\Gamma(M,E^\pm)\to \Gamma(M,E^\mp)$. We set $D^\pm:= D|_{\Gamma(M,E^\pm)}$. Then 
\[
    D\ = \ \begin{pmatrix}
        0&D^-\\D^+&0
    \end{pmatrix}
\]
with respect to the decomposition $E=E^+\oplus E^-$. Our assumptions imply that  $\tau\, D^\pm\, \tau^{-1}= (D^\pm)^*= D^\mp$. 

Suppose now that $D$ is Fredholm. In particular, this is always the case when $M$ is compact.  Then the $\tau$-index $\ind_\tau D^+\in \ZZ_2$ is defined by \eqref{E:tauindex}.
The goal of this paper is to study the properties of this index. 

%--------
\begin{definition}\label{D:Sobolev space}
We denote by $W^{1,2}(M,E)$ the domain of the closure of $D$ viewed as a Hilbert space (this is the space which we denoted by $W$ in Section~\ref{S:tauindex}). We refer to $W^{1,2}(M,E)$ as the \textit{Sobolev space}.
\end{definition}
The operator $D$ is a bounded operator $D:W^{1,2}(M,E)\to L^2(M,E)$. 

%----
\begin{lemma}\label{L:compactperturbation}
Let $D$ be a Fredholm odd symmetric elliptic differential operator on a graded quaternionic bundle $E$. Suppose $V:E\to E$ is an odd symmetric bundle map which is odd with respect to the grading, $V:E^\pm\to E^\mp$. Assume that the support of $V$ is compact in $M$. Then $D+V$ is Fredholm and
\[
    \ind_\tau (D^++V^+)\ = \ \ind_\tau D^+, 
\]
where $V^+$ denotes the restriction of $V$ to $E^+$.
\end{lemma}
\begin{proof}
The Fredholmness follows from \cite[Theorem~2.1(vi)]{Anghel93}.  The operator $V:W^{1,2}(M,E)\to L^2(M,E)$ is a compact by Rellich’s Lemma. The lemma follows now from Corollary~\ref{C:compact perturbation}.
\end{proof}

Next, we consider several examples of Dirac-type operators with non-trivial $\tau$-index. 

%---------------------------------------------
\subsection{Example 1: An operator on a real line}\label{SS:exampleline}
This example is a reformulation of the example in Section~\ref{SS:DAline}. Let $M=\RR$ and let $\theta(t):= -t$. Then $(M,\theta)$ is an involutive manifold of dimension 1. We identify the tangent space to  $M$ at any point $t$ with $\RR$. Let $E:= M\times \CC^2$ and consider the operator $D^+:= i\frac{\p}{\p t}$.

Set 
\[
    V^+ \ := \ \begin{pmatrix}
        i\arctan(t)&0\\ 0&-i\arctan(t)
    \end{pmatrix}.
\]

Define the anti-involution $\theta^E$ by $\theta^E(z_1,z_2)= (\bar{z}_2,-\bar{z}_1)$. Then, 
\[
    \tau:\, \big(\sigma_1(t),\sigma_2(t)\big)\ \mapsto \
    \big(\bar{\sigma}_2(-t),\bar{\sigma}_1(-t)\big).
\]
One readily checks that the operator $D_V^+= D+V^+$ is $\tau$-symmetric. 

Then $D_V^+=iD$ where $D$ is given by \eqref{E:DC2}. Hence, by Lemma~\ref{L:indexnormalization}, we have
\[
    \ind_\tau D_V^+ \ \equiv \ \dim\ker D_V^+\ = \ 1 
\]
where $\equiv$  denotes equality modulo 2.
%---------------------------------------------
\subsection{Example 2: A torus with the trivial involution} \label{SS:torustricvail}
Let $M=S^1\times S^1$ be a two-dimensional torus with the product metric. We write a point $x\in M$ as $x=(e^{it},e^{is})$, where $t,s\in [-\pi,\pi]$ and let $\theta:M\to M$ be the identity map. Let $E= M\times\CC^2$ be the trivial bundle over $M$, endowed 
with the trivial metric and connection.  We identify sections of $E$ with two-component functions $(f_1,f_2):M\to \CC^2$.

%-----
\begin{comment}
Define the Clifford action of $TM$ on $E$ by 
\begin{equation}\notag
    c\big(\frac{\p}{\p t}\big)\ := \
    \begin{pmatrix}
        0&1\\
       -1&0
    \end{pmatrix},\qquad
    c\big(\frac{\p}{\p s}\big)\ := \
    \begin{pmatrix}
      0&i\\
    i&0
    \end{pmatrix}.
\end{equation} 
\end{comment}
%------
Consider the operator 
\[
    D\ := \ 
    \begin{pmatrix}
        0& -\frac{\p}{\p t}+i\frac{\p}{\p s}\\
        \frac{\p}{\p t}+i\frac{\p}{\p s}&0
    \end{pmatrix}.
\]
Then $D^+= \frac{\p}{\p t}+i\frac{\p}{\p s}$.

Define an anti-unitary anti-involution $\theta^E:M\times\CC^2\to M\times \CC^2$ by $\theta^E:(x;z_1,z_2)= (x;\bar{z}_2,-\bar{z}_1)$, where $x\in M,\ z_1,z_2\in \CC^2$. Then $\tau\big(f_1(x),f_2(x)\big)= \big(\bar{f}_2(x),-\bar{f}_1(x)\big)$. One readily sees that $\tau$ commutes with $D$  and, hence,  $D$ is odd symmetric. Then the kernel of $D^+$ consists of functions $f:M\to \CC$ such that 
\[
    \frac{\p f}{\p t}\ + \ i\frac{\p f}{\p s} \ = \ 0.
\]
It follows that $f$ is harmonic on $M=S^1\times S^1$: 
\[
    \frac{\p^2 f}{\p t^2}\ + \ \frac{\p^2 f}{\p s^2}
    \ = \ 0.
\]
Hence, $f$ is a constant, i.e. the kernel of $D^+$ is one-dimensional and 
\[
    \ind_\tau D^+\ \equiv \ \dim\ker D^+ \ = \ 1\ \in \ZZ_2.
\]
%---------------------------------------------
\subsection{Example 3: A a torus with time-reversal involution}\label{SS:torustimereversal}
Our next example is a reformulation of the bulk index of 2d Topological Insulator of type AII, cf. \cites{GrafPorta13,Hayashi17}.

Let $M= \CC/\{m+in:m,n\in\ZZ\}$ be a complex torus. Topologically it is $S^1\times S^1$. We write a point in $M$ as $z=s+it$, $(-\pi\le t,s\le \pi)$ and define a holomorphic involution $\theta:M\to M$, by $\theta:s+it\mapsto -s-it$.

Let $L_n$ be a complex line bundle over $M$ with Chern class $n\in \ZZ\simeq H^2(M,\ZZ)$. It can be described as the quotient of the  trivial bundle $[-\pi,\pi]\times[-\pi,\pi]\times\CC$ over $[-\pi,\pi]\times[-\pi,\pi]$ by the equivalence relations
\[
    (-\pi,s,v) \ \sim \ (\pi,s,v), \qquad 
    (t,-\pi,v)\ \sim \ (t,\pi,ve^{int}).
\]
To see that the Chern class of $L_n$ is indeed $n$, let us consider the connection $\n$ on $L_n$, which at the point $s+it\in M$ has the form
\[
    \n_{\frac{\p}{\p s}} \ = \ \frac{\p}{\p s}, \qquad
    \n_{\frac{\p}{\p t}} \ = \ 
    \frac{\p}{\p t}-\frac{in}{2\pi}s.
\]
Notice that at $s=\pm\pi$ we have $\n_{\frac{\p}{\p t}}= \frac{\p}{\p t}\mp\frac{in}{2}$. Hence, 
\[\n_{\frac{\p}{\p t}}\big|_{s=-\pi} \ = \ 
    \frac{\p}{\p t}+\frac{in}{2} \ = \ 
    e^{-int}\,\n_{\frac{\p}{\p t}}\big|_{s=\pi}\,e^{int},
\]
which shows that $\n$ is a smooth connection on $L_n$. The curvature of $\n$ is the form $\omega= -\frac{in}{2\pi}dt\wedge ds$.  Thus the representative of the Chern class of $L_n$ in Chern-Weil theory is $\frac{i}{2\pi}\omega= \frac{n}{(2\pi)^2}dt\wedge ds$. Since the integral of this form over $M$  is $n$, we conclude that $L_n$ is indeed a bundle with Chern class $n$. 

The curvature of the pull-back bundle $\theta^*L_n$ is also equal to $\omega= \theta^*\omega$.  So $\theta^*L_n$ is also a line bundle with Chern class $n$ and, hence, is isomorphic to $L_n$.  

We view the bundle $L_{-n}$ as the complex conjugate of $L_n$ and define a bundle map $\alpha_n: L_n\to \theta^*L_{n}$ by 
\[
    \alpha_n:\, (t,s,v)\ \mapsto \ (-t,-s,\bar{v}), \qquad -\pi\le t,s\le \pi.
\]
One readily sees that this formula defines a smooth map over $M$ (i.e., satisfies the correct periodic conditions on the boundary).  It is an anti-linear map and  $\alpha_n\circ\alpha_{-n}= \ID$.

Let $\Lambda^{0,j}M$ ($j=0,1$) denote the bundle of anti-holomorphic exterior forms on $M$. The space of sections $\Gamma(M,\Lambda^{0,j}M\otimes L_n)$ is the space $\Omega^{0,j}(M,L_n)$ of anti-holomorphic differential forms with values in $L_n$. Set
\[
      E^+\ := \ \big(\Lambda^{0,0}M\otimes L_n\big)\oplus 
    \big(\Lambda^{0,0}M\otimes L_{-n}\big);\qquad
    E^-\ := \ \big(\Lambda^{0,1}M\otimes L_{-n}\big)\oplus 
    \big(\Lambda^{0,1}M\otimes  L_n\big).
\]
Thus the space of sections of $E=E^+\oplus E^-$ is equal to $\Omega^{0,*}(M,L_n\oplus L_{-n})$.

Let $c:TM\to \End \Lambda^{0,*}M$ be the standard Clifford action. It is defined as follows (see \cite[\S2]{BrBackground} for more details): for a tangent vector $v\in TM$ we write it as $v= v^{1,0}\oplus v^{0,1}$ according to the decomposition $T_\CC M= T^{1,0}M\oplus T^{0,1}M$ of the tangent bundle to the holomorphic and anti-holomorphic part. Using the Riemannian metric we transform $v^{1,0}$ to a cotangent vector $Iv^{1,0}$. Then 
\[
    c(v):\,\alpha \ \mapsto \ \sqrt{2}\,\big(\,
    Iv^{1,0}\wedge \alpha\ - \ \iota_{v^{0,1}}\alpha\,\big).
\]
Then $c(v)^2=-|v|^2$.  We also denote by $c(v)$ the induced operators on $\Lambda^{0,*}M\otimes\big(L_n\oplus L_{-n}\big)$ and $\Omega^{0,*}(M,L_n\oplus L_{-n})$.

Since $\p/\p s= \frac{\p/\p z+\p/\p\bar{z}}2$ and $\p/\p t= \frac{\p/\p z-\p/\p\bar{z}}{2i}$, we have 
\begin{equation}\label{E:cds}
\begin{aligned}
           c(\p/\p s):\, \alpha\ &\mapsto  \ \frac1{\sqrt{2}}\,\big(\,
    d\bar{z}\wedge \alpha\ - \ \iota_{\p/\p \bar{z}}\alpha\,\big). \\
     c(\p/\p t):\, \alpha\ &\mapsto  \ \frac1{i\sqrt{2}}\,\big(\,
    -d\bar{z}\wedge \alpha\ - \ \iota_{\p/\p \bar{z}}\alpha\,\big).
\end{aligned}
\end{equation}

We define an anti-linear anti-involution on  $E$ by 
\[
    \theta\ := \  \begin{pmatrix}
        0&\alpha_{-n}\\\alpha_{n}&0
    \end{pmatrix}\circ c(\p/\p s)
\]
and let $\tau: \Omega^{0,*}(M,L_n\oplus L_{-n})\to \Omega^{0,*}(M,L_n\oplus L_{-n})$ be as in \eqref{E:tau=}.

Consider the Dolbeault-Dirac operators 
\[
    D_{\pm n}\ := \ c(\p/\p s)\, \frac{\p}{\p s}\ + \ c(\p/\p t)\, \frac{\p}{\p t}:\,
    \Omega^{0,*}(M,L_{\pm n})\ \to \  \Omega^{0,*}(M,L_{\pm n}).
\]

By \eqref{E:cds}, $c(\p/\p s)$ commutes and $c(\p/\p t)$ anti-commutes with complex conjugation and, hence, with $\begin{psmallmatrix}
0&-\alpha_n\\\alpha_{n}&0\end{psmallmatrix}$. 
Also $c(\p/\p s)$ commutes and $c(\p/\p t)$ anti-commutes with $c(\p/\p s)$. Therefore, 
\[
    \tau\circ  D_{\pm n}\circ \tau^{-1} \ = \ D_{\mp n}.
\]
In particular, setting $D_{\pm n}^+:= D_{\pm n}\big|_{\Omega^{0,0}(M,L_{\pm n})}$ and $D_{\pm n}^-:= D_{\pm n}\big|_{\Omega^{0,1}(M,L_{\pm n})}$, we obtain
\begin{equation}\label{E:kerDpmn}
    \ker D_{-n}^+  \ = \ \ker D_{n}^-.
\end{equation}

Set 
\[
    D\ := \ D_n\oplus D_{-n}:\,\Omega^{0,*}(M,L_n\oplus L_{-n})\ \to \  \Omega^{0,*}(M,L_n\oplus L_{-n}).
\]
Then $\tau\circ D\circ \tau^{-1}= D$, i.e., $D$ is odd-symmetric. By \eqref{E:kerDpmn},
\begin{multline}
     \ind_\tau D^+\ = \ \dim\ker D^+ \ = \ \dim\ker D_n^+ \ + \ 
     \dim\ker  D_{-n}^-
    \\ = \  \dim\ker D_{n}^+ \ + \ \dim\ker  D_{n}^-   \equiv \ \ind D_n, 
\end{multline} 
where the last equality holds modulo 2. 

Since the Todd class of $M$ is trivial, it follows from the  Atiyah-Singer index theorem that the index of $D_n$  is equal to the Chern class of $L_n$.  We conclude that\textbf{ $\ind_\tau D^+= 0\in \ZZ_2$ } if $n$ is even and $\ind_\tau D^+= 1\in \ZZ_2$ if $n$ is odd.

%------------------------------------------------------
%------------------------------------------------------
\section{The relative index theorem}\label{S:relative index}

In this section, we study the  $\tau$-index of odd symmetric  Callias-type operators on non-compact manifolds and prove a  $\ZZ_2$-valued analog of the Gromov-Lawson relative index theorem. In the next section, we prove an analog of the Callias-type index theorem. 

This section's structure and many constructions are borrowed from \cite{BrCecchini17}. 

%-----------------------------------------------------
\subsection{Callias–type operators}\label{SS:Calliasoperators}
Consider a complete Riemannian manifold $M$ and let $E=E^+\oplus E^-$ be a graded complex vector bundle over $M$. We consider a formally self-adjoint  first order elliptic differential operator  $D$  
\[
D\ = \ \begin{pmatrix}
    0&D^-\\D^+&0
\end{pmatrix}
\]
acting on smooth sections of $E$. Here $D^-$ and $D^+$ are differential operators associated with the negative and positive parts of $E$, respectively. We make the following

\begin{assumption}\label{A:bounded principal symbol}
The principal symbol of $D$ is uniformly bounded from above, i.e. there exists a constant $b>0$ such that
\begin{equation}\label{E:bounded principal symbol}
    \big\|\sigma(D)(x,\xi)\big\| \ \le \ b\,|\xi|, 
    \qquad \text{for all}\quad x\in M,\ \xi\in T^*_xM.
\end{equation}
Here $|\xi|$ denotes the length of $\xi$ defined by the Riemannian metric 
,\/ $\sigma(D)(x,\xi): E^{\pm}_x\rightarrow E^{\mp}_x$ is the leading symbol of $D$, and $\|\sigma(D)(x,\xi)\|$ is its operator norm. 
\end{assumption}

Under this assumption, the operator $D$ is essentially self-adjoint by  \cite{GromovLawson83}.
We note that  Assumption~\ref{A:bounded principal symbol} is automatically satisfied if $D$ is a Dirac-type operator. 

Let $F^+:E^+\rightarrow E^-$ be a bundle map. Let $F^-= (F^+)^*$ be the formal adjoint of $F^+$ and set
\[
    F \ := \ \left(\begin{array}{cc}0&F^-\\F^+&0\end{array}\right).
\]

\begin{definition}\label{D:admissible endomorphism}
The endomorphism $F$ is called \emph{admissible} if there exists a compact set $K\subset M$ and a constant $c>0$ such that
\begin{enumerate}
\item[(i)] the restriction of the \textit{anticommutator} 
\[
    \{D,F\} \ := \ D\, F\ + \ F\, D \ = \
    \begin{pmatrix}
        D^-F^++F^-D^+&0\\0&D^+F^-+F^+D^-
    \end{pmatrix}
\]
 to $M\setminus K$ is an operator of order $0$, i.e. a bundle map;
\item[(ii)] for every $x\in M\setminus K$\/ we have
\begin{equation}\label{E:callias condition}
    F(x)^2\ \ge \ \big\|\{D,F\}(x)\big\|\ +  \ c,
\end{equation}
where $\big\|\{D,F\}(x)\big\|$ denotes the operator norm of the bounded operator $\{D,F\}(x):E_x\to E_x$.
\end{enumerate}
In this case, we refer to the compact set $K$ as an \emph{essential support} of $F$.
\end{definition}

%---
\begin{definition}\label{D:Callias}
A (graded) \textbf{Callias-type operator} is an operator of the form
\begin{equation}\label{E:Callias}
	\B= \B_F\ :=\ D\ +\ F\ =\ \begin{pmatrix}
	    0&D^-+F^-\\D^++F^+&0
	\end{pmatrix},
\end{equation}
where $F$ is an admissible endomorphism. We set $\B^\pm:=D^\pm+F^\pm$.
\end{definition}

A Callias-type operator is Fredholm, cf. \cite[Proposition~1.4]{Anghel93Callias}.

Suppose now that we are given an involution $\theta:M\to M$ and that  $(E,\theta^E)$ is a graded quaternionic vector bundle over $(M,\theta)$, cf. Section~\ref{SS:quaternionic}. Assume that $D$ and $F$ are odd symmetric. Then $\B$ is an odd symmetric Callias-type operator, and its $\tau$-index  $\ind_\tau \B$ is defined. The rest of this section is dedicated to the study of this index. 

%--------
\begin{lemma}\label{L:essential support}
If the essential support of $F$ is empty, then $\ind_\tau \B= 0$.
\end{lemma}
\begin{proof}
If \eqref{E:callias condition} holds everywhere on $M$, then 
\[
    \B^2\ = \  D^2\ + \ \{D,F\}\ + \ F^2 \ \ge \ D^2\ + \ c \ > \ 0. 
\]
Hence, $\ker \B=\{0\}.$
\end{proof}

%-----------------------------------------------------
\subsection{The relative index theorem}\label{SS:relativeindex}
Suppose $(E_j,\theta^{E_j})$ ($j=0,1$) are graded quaternionic bundles over complete involutive Riemannian manifolds $(M_j,\theta_j)$ and $\B_j=D_j+F_j$ are odd symmetric Callias-type operators. Let $\tau_j:\Gamma(M_j,E^j)\to \Gamma(M_j,E^j)$ be as in \eqref{E:tau=}. 

We say that a bundle map $\Psi:E_j\to E_k$  ($j,k=0,1)$ is \textit{$\tau$-equivariant} if $\Psi\,\tau_j= \tau_k\,\Psi$. We say that $\Psi$ is graded if $\Psi:E_j^\pm\to E_k^\pm$.

Suppose there are  partitions $M_j=W_j\cup_{N_j} V_j$ into $\theta_j$-invariant submanifolds, such that  $N_j$ are \emph{compact} $\theta_j$-invariant hypersurfaces. 
We make the following
\begin{assumption}\label{A:cut and paste}
There exist $\theta_j$-invariant tubular neighborhoods $U(N_j)$ of $N_j$ ($j=0,1$) and a $\tau$-invariant  diffeomorphism $\psi:U(N_0)\rightarrow U(N_1)$ satisfying the following properties:
\begin{enumerate}
\item[(i)] The restriction of $\psi$ to $N_0$ is a diffeomorphism onto $N_1$;
\item[(ii)] there exists a $\tau$-equivariant isomorphism of graded quaternionic bundles $\Psi:E_0|_{U(N_0)}\rightarrow E_1|_{U(N_1)}$ covering $\psi$;
\item[(iii)] the restrictions $D_0|_{U(N_0)}$ and $D_1|_{U(N_1)}$ are conjugated through the isomorphism $\Psi$.
\end{enumerate}
\end{assumption}

We cut  $M_j$ along $N_j$ and use the map $\psi$ to glue the pieces together interchanging $V_0$ and $V_1$, cf. Figure~\ref{fig:cut and paste}
In this way, we obtain the manifolds 
\[
	M_2:=W_0\cup_NV_1, \qquad M_3:=W_1\cup_NV_0,
\]
where $N\cong N_0\cong N_1$. We refer to $M_2$ and $M_3$ as {\em manifolds obtained from $M_0$ and $M_1$ by cutting and pasting}. They are naturally involutive manifolds. 
\begin{figure}[ht]
    \centering
    \vskip-2cm
    \includegraphics[width=0.90\textwidth]{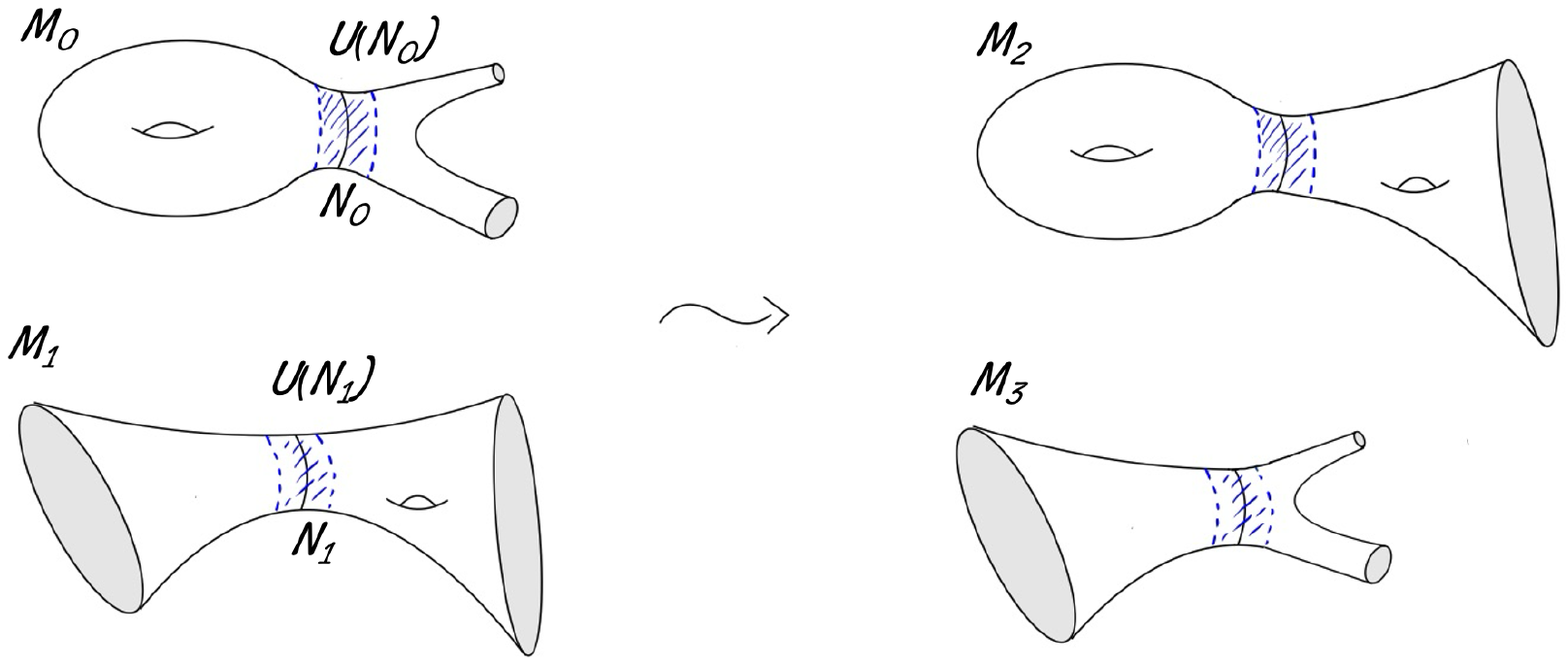}
    \vskip-2.5cm\caption{Obtaining $M_2$ and $M_3$ by cutting and pasting.}
    \label{fig:cut and paste}
\end{figure}

We use the map $\Psi$ to cut the bundles $E_0$, $E_1$ at $N_0$, $N_1$ and glue the pieces together interchanging $E_0|_{V_0}$ and $E_1|_{V_1}$.
With this procedure, we obtain graded quaternionic bundles $E_2\rightarrow M_2$ and $E_3\rightarrow M_3$.
Using Condition~(iii) of Assumption~\ref{A:cut and paste} we construct a formally self-adjoint first order elliptic differential operator $D_2$ acting on smooth sections of $E_2$ satisfying Assumption~\ref{A:bounded principal symbol} and such that $D_2|_{W_0}=D_0|_{W_0}$ and $D_2|_{V_1}=D_1|_{V_1}$.
Similarly, define an elliptic operator $D_3$ acting on smooth sections of the bundle $E_3$. The operators $D_2$ and $D_3$ are odd symmetric. 

Assume $F_0$ and $F_1$ are admissible odd symmetric endomorphisms of the bundles $E_0$ and $E_1$ respectively. For $j=0,1$ we choose positive $\theta$-invariant cut-off functions $\alpha_j,\,\beta_j\in C^\infty(M_j)$ such that:
\begin{enumerate}
 \item \label{C:cut-off1} $\supp\alpha_j\subset W_j\cup U(N_j)\textrm{ and }\supp\beta_j\subset V_j\cup U(N_j)$;
\item \label{C:cut-off2} $\alpha_j=1\textrm{ on }W_j\setminus U(N_j)\textrm{ and }\beta_j=1\textrm{ on }V_j\setminus U(N_j)$;
\item \label{C:cut-off3} $\alpha_j^2+\beta_j^2=1$.
\end{enumerate}
By construction, $E_2|_{W_0\cup\, U(N_0)}\cong E_0|_{W_0\cup \,U(N_0)}$ and $E_2|_{V_1\cup\, U(N_1)}\cong E_1|_{V_1\cup\, U(N_1)}$.
We use this identification and Condition~(\ref{C:cut-off1}) to define the endomorphism $F_2:=F_0\alpha_0+F_1\beta_1$ of $E_2$. In the same way, we define the endomorphism $F_3:=F_1\alpha_1+F_0\beta_0$ of $E_3$.
Observe that $F_2$ and $F_3$ are admissible so that the operators $\B_2:=D_2+F_2$ and $\B_3:=D_3+F_3$ are of Callias-type. We refer to $\B_2$ and $\B_3$ as {\em operators obtained from $\B_0$ and $\B_1$ by cutting and pasting}.

The first result of this section is the following

\begin{theorem}[\textbf{Relative $\ZZ_2$-valued index theorem}]\label{T:relative index}
\[
	\ind_\tau \B_0^+\ + \ \ind_\tau \B_1^+\ =\ 
    \ind_\tau \B_2^+\ +\ \ind_\tau \B_3^+.
\]
\end{theorem}

\begin{remark}
For the usual $\ZZ$-valued index this theorem was proved by Gromov and Lawson in \cite{GromovLawson83}. A $K$-theoretical version has been proved by Bunke in \cite{Bunke95}. Our formulation of the relative index theorem is close to this last one.
\end{remark}

%------------------------------------------------------------------------
\subsection{The proof of Theorem~\ref{T:relative index}}
Set $M:=M_1\sqcup M_2\sqcup M_3\sqcup M_4$ and let $\theta$ be the involution on $M$ induced by $\theta_j$ $(j=0,1,2,3,4)$.
We use $E_1$, $E_2$, $E_3$ and $E_4$ to construct the bundle
\begin{equation}\label{E:bundleEoplus}
E:=E_1\oplus E_2\oplus E_3^\textrm{op}\oplus E_4^\textrm{op}
\end{equation}
over $M$, where the superscript ``op" means that we consider the opposite grading on the fibers of $E_3$ and $E_4$. Then $\theta^{E_j}$ induces an anti-involution $\theta^E$ on $E$, which gives $E$ a structure of a graded quaternionic vector bundle over $M$.

Consider the Callias-type operator 
\begin{equation}\label{E:decomposition of B}
\B \ :=\ \B_1\oplus \B_2\oplus \B_3\oplus \B_4
\end{equation}
acting on smooth sections of $E$. Our grading agreement \eqref{E:bundleEoplus} implies that 
\[
    \B^+\ = \ \B_1^+  \oplus \B_2^+ \oplus \ \B_3^-\oplus B_4^-.
\]
Since the operators $\B_j$ are odd symmetric we have $\dim\ker \B_j^+= \dim\ker \B_j^-$. Thus we have the following chain of equalities in $\ZZ_2$:
\begin{multline}\label{E:tau-index of B}
    \ind_\tau \B^+\ := \ \dim\ker \B^+\ = \ 
    \dim\ker \B_1^+\ + \ \dim\ker \B_2^+\ + \ \dim\ker \B_3^-\ + \ \dim\ker \B_4^+
    \\ = \ 
    \ind_\tau \B_1^++\ind_\tau \B_2^+\ + \ \ind_\tau \B_3^+ \ + \ \ind_\tau \B_4^+\ \in \ \ZZ_2.
\end{multline}
Hence, Theorem~\ref{T:relative index} is equivalent to the equality $\ind_\tau \B^+=0$.

For $j=1,2$, let $\alpha_j$, $\beta_j\in C^\infty(M_j)$ be the  cut-off functions defined in Section ~\ref{SS:relativeindex}. Using these function in \cite[\S3.3]{BrCecchini17}  the authors constructed a self-adjoint bundle map $U\in C^\infty(M,\End(E))$ with the following properties:
\begin{enumerate}%[label=\textbf{(U.\arabic*)},ref=U.\arabic*]
\item \label{U:property1}$U^2=\ID_E$;
\item\label{U:odd} $U:E^\pm\to E^\mp$ is odd with respect to the grading (to have this property we need to define the grading in $E$ by \eqref{E:bundleEoplus});
\item \label{U:property2} the anticommutator $\{\B,U\}:=\B U+U\B$ is an even compactly supported endomorphism of the bundle $E$.
\end{enumerate}
This operator is given by an explicit formula in terms of the cut-off functions $\alpha_j$, $\beta_j\in C^\infty(M_j)$ ($j=0,1$). One readily sees that, since these functions are $\theta_j$-invariant, the operator $U$ is odd symmetric. Thus so is the operator $U\B U$. By Condition~(\ref{U:property1}), $U^{-1}=U$. \\

From Condition~(\ref{U:property2}) we see that $W:= U\B U+\B$ is an odd symmetric bundle map with compact support. Consider the operator
\begin{equation}\label{E:B+UB U}
    \frac{\B-U\B U}{2} \ = \ 
    \B\ - \ \frac{W}{2}.
\end{equation}

Let $(\B-U\B U)^+$ denote the restriction of $\B-U\B U$ to $E^+$. By Lemma~\ref{L:compactperturbation}, 
\[
    \ind_\tau \B^+\ = \ \ind_\tau \frac{(\B-U\B U)^+}{2}.
\]
Since $U$ is odd symmetric and self-adjoint, $U\,\tau= \tau\, U$. Hence,  the map $\tau\, U:\Gamma(M,E)\to \Gamma(M,E)$ is an anti-unitary anti-involution of even grading degree, i.e. $\tau\, U:\Gamma(M,E^+)\to \Gamma(M,E^+)$.  Since $\B$  and U are odd symmetric, we have
\[
    (\tau\, U)\,\frac{(\B-U\B U)^+}{2}\, (\tau\, U)^{-1}\ 
    = \ -\,\frac{(\B-U\B U)^+}{2}.
\]
Hence, $\tau\, U$ acts on the kernel of $(\B-U\B U)^+$.  By Lemma~\ref{L:even},  the dimension of the kernel of $(\B-U\B U)^+$ is even. Thus, by \eqref{E:B+UB U} 
\[
    \ind_\tau \B\ = \ \ind_\tau \frac{(\B-U\B U)^+}{2}\ \equiv \ 
    \dim\ker \frac{(\B-U\B U)^+}{2} \ = \ 0 \ \in \ \ZZ_2,
\]
as required. \hfill $\square$

%------------------------------------------------------
%------------------------------------------------------
\section{The  Callias-type theorem}\label{S:callias}

In this section, we consider the case of a Dirac-type Callias-type operator and prove a $\ZZ_2$-valued analog of the Callias index theorem, \cites{Anghel93Callias,Bunke95}.

%---------------------------------------------
\subsection{A (generalized) Dirac operator}\label{SS:Dirac}
 Suppose now that $E$ is a {\em Dirac bundle} over a complete oriented Riemannian manifold $M$, cf. \cite[Definition~II.5.2]{LawMic89}. By this, we mean that it is endowed with  {\em Clifford action} $c:TM\to \End(E)$, such that $c(v)^2=-|v|^2$ ($v\in TM$), a Hermitian metric $h^E$ with respect to which the operators $c(v)$  are skew-adjoint, and a Hermitian {\em Clifford connection} $\n^E$ which satisfies the \textit{Leibniz rule} 
 \begin{equation}\label{E:Leibniz}
     \big[\n^E,c(v)\big] \ = \  c(\n^{LC}v),
 \end{equation}
 where $\n^{LC}$ denote the Levi-Civita connection on $TM$. Then, the {\em (generalised) Dirac operator} is defined by 
\begin{equation}\label{E:Dirac}
     \dirac\ := \ \sum_{1\le j\le n}\, c(e_i)\n^E_{e_j}:\,\Gamma(M,E)\to \Gamma(M,E), 
\end{equation}
where $(e_1,\ldots,e_n)$ is an orthonormal frame of $TM$. Then $\dirac$ is a formally self-adjoint differential operator, which is independent of the choice of the frame $(e_1,\ldots,e_n)$, cf. \cite[\S3.3]{BeGeVe}.  

Consider the operator 
\begin{equation}\label{E:Gamma}
    \Gamma\ := \  i^{[\frac{n+1}{2}]} c(e_1)\cdot c(e_2)\cdots c(e_n),
\end{equation}
where $[x]$ denotes the integer part of $x$. Then, \cite[Lemma~3.17]{BeGeVe},  $\Gamma$ is an involution $\Gamma^2= 1$ which is independent of the choice of the frame $(e_1,\ldots,e_n)$. 

Consider, first, the case when $n$ is even. Then, \cite[Lemma~3.17]{BeGeVe}, $\Gamma$ anti-commmutes with $c(v)$. Let $E^+$ and $E^-$ denote the eigenspaces of $\Gamma$ with eigenvalues 1 and $-1$ respectively. Then $E=E^+\oplus E^-$ and $c(v):E^\pm\to E^\mp$. The space of sections of $E$ is also graded by 
\[
    \Gamma(M,E)\ = \ \Gamma(M,E^+)\oplus \Gamma(M,E^-),
\]
and the Clifford connection $\n^E$ preserves the grading. It follows that the Dirac operator $\dirac$ is {\em odd with respect to the grading}, i.e., $\dirac:\Gamma(M,E^\pm)\to \Gamma(M,E^\mp)$. We denote by $\dirac^\pm$ the restriction of $\dirac$ to $\Gamma(M,E^\pm)$. Then, $\dirac^\pm$ is the formal adjoint of $\dirac^\mp$, and, with respect to the grading, the Dirac operator has the form 
\begin{equation}\label{E:Diracgrading}
    \begin{pmatrix}
        0&\dirac^-\\ \dirac^+&0
    \end{pmatrix}.
\end{equation}

If $n$ is odd, then $\Gamma$ commutes with $c(v)$ ($v\in TM$) and does not define a useful grading on $E$. So we consider $E$ as an ungraded vector bundle.

%---------------------------------------------
\subsection{An odd symmetric Dirac operator}\label{SS:oddDirac}
Let  $(E,\theta^E)$ be a quaternionic bundle over an involutive manifold $(M,\theta)$, which is also a Dirac bundle (no connections between the two structures on $E$ yet). If $n$ is even, we assume, in addition, that $\theta^E$ is odd with respect to the grading.

Let $\tau$ be as in \eqref{E:tau=}. Suppose $\dirac$ is odd symmetric,  $\tau\, \dirac\, \tau^{-1}= \dirac$. If $n$ is even, that means that $\dirac^\pm= \tau\, \dirac^\mp\, \tau^{-1}$.

%---------------------------------------------
\subsection{A relationship between the quaternionic and the Clifford structures}\label{SS:tau-c}
Let $\theta_*:T_xM\to T_{\theta x}M$ denote the differential of $\theta$. Since $\theta$ is an involution, the $\theta_*= \theta_*^{-1}= \theta^*$, where $\theta^*$ is the pull-back map of on the cotangent bundle, which we identified with $TM$.   Comparing the leading symbols of $\dirac$ and  $\tau\,\dirac\, \tau^{-1}$   we conclude that
\begin{equation}\label{E:c-tau}
       \theta^E \, c(\theta_*\xi)\,(\theta^E)^{-1}\ = \ c(\xi).
\end{equation}
One then readily sees that the equation $\tau\dirac\tau^{-1}=\dirac$ implies that 
\begin{equation}\label{E:n-tau}
    \theta^E\,\n^E_{\theta_*\xi}\,   (\theta^E)^{-1}\ = \ \n^E_\xi.
\end{equation}
      
%---------------------------------------------
\subsection{A potential}\label{SS:oddpotential}
Let $V:E\to E$ be an odd symmetric bundle map. If $n$ is even, we assume that $V$ is odd with respect to the grading, i.e.
\begin{equation}\label{E:F=}
        V\ = \ 
    \begin{pmatrix}
        0&V^-\\ V^+&0
    \end{pmatrix}.
\end{equation}

Consider the Dirac-type operator (or \textit{generalized Dirac operator with potential})
\begin{equation}\label{E:Diractype} 
    D_V^+\ := \ \begin{cases}
        \dirac^+\ + \ V^+, \qquad&\text{if $n$ is even}\\
        \dirac\ + \ V, \qquad&\text{if $n$ is odd}.
    \end{cases}
\end{equation}
Then $D_V^+$ is odd symmetric. 

We have a ``+" in the notation for the Dirac-type operator in the odd-dimensional case to be able to study both even and odd-dimensional cases simultaneously. To this end, if $n$ is odd,  we set $D_V^-:= (D_V^+)^*$ and consider the graded operator 
\[
    D_V\ := \ \begin{pmatrix}
        0&D_V^-\\ D_V^+&0
    \end{pmatrix}
    \ = \ \begin{pmatrix}
        0&\dirac+V^*\\ \dirac+V&0
    \end{pmatrix}:\, \Gamma(M,E\oplus E)\ \to \ \Gamma(M,E\oplus E).
\]

%----------------
\subsection{The $\tau$-index of an odd symmetric Dirac-type operator}\label{SS:indexDirac}
Suppose now that $D_V$ is Fredholm. In particular, this is always the case when $M$ is compact.  Then the $\tau$-index $\ind_\tau D_V^+\in \ZZ_2$ is defined by \eqref{E:tauindex}.

If $M$ is compact, then it follows from Lemma~\ref{L:compactperturbation}, that $\ind_\tau D_V^+= \ind_\tau\dirac^+$. But in this section, we will be mostly interested in the non-compact case. However, the following analog of the classical result for the usual $\ZZ$-valued index will play a role:

%------------
\begin{lemma}\label{L:vanishing on compact}
Suppose $M$ is a compact involutive manifold of odd dimension and let $D_V^+=\dirac+V$ be an odd symmetric Dirac-type operator on a quaternionic bundle $(E,\theta^E)$  over $M$. Then $\ind_\tau D_V^+= 0$.
\end{lemma}
\begin{proof}
It suffices to show that $\ind_\tau\dirac=0$. Since $\dirac$ is odd symmetric and self-adjoint, it commutes with $\tau$.  It follows that $\tau$ acts on $\ker\dirac$. Hence, by Lemma~\ref{L:even}, the dimension of the kernel of $\dirac$ is even. Thus 
\[
    \ind_\tau D_V^+\ =\ \ind_\tau\dirac\ \equiv \ \dim\ker\dirac \ = \ 0\ \in \ \ZZ_2.
\]
\end{proof}

\begin{comment}
Recall that the usual $\ZZ$-valued index of any elliptic differential operator on an odd-dimensional manifold vanishes. The simplest proof of it is by application of the Atiyah-Singer index theorem, which is not available for the $\ZZ_2$-valued index.  We, however, conjecture that  $\tau$-index does vanish for any elliptic operator, the $\tau$-index of any odd symmetric elliptic differential operator on an odd-dimensional manifold vanishes.

%---------------------
\begin{conjecture}
    The $\tau$-index of any odd symmetric elliptic differential operator on a compact involutive odd-dimensional manifold vanishes. 
\end{conjecture}
\end{comment}
\medskip

Next, we consider several examples of the Dirac-type operators with non-trivial $\tau$-index. 
%------------------------------------------------------------------------
\subsection{A ungraded Callias-type operator of Dirac-type}\label{SS:CalliasDirac}
Let $(M,\theta)$ be a complete involutive manifold \textbf{of odd dimension} $n$. Let $(E,\theta^E)$ be a quaternionic Dirac bundle over $(M,\theta)$ and let $\dirac$ be the corresponding generalized Dirac operator. We assume that $\dirac$ is odd symmetric. Let $V:E\to E$ be an odd symmetric bundle map and consider the odd symmetric Dirac-type operator
\begin{equation}\label{E:D+V Callias}
        D^+_V \ = \ \dirac \ + \ V:\,\Gamma(M,E)\to \Gamma(M,E).
\end{equation}
As before, we set $D^-_V:= (D^+_V)^*$ and consider the operator $D_V= \begin{psmallmatrix}
  0& D^-_V\\
  D^+_V & 0
\end{psmallmatrix}$. 

Let $\Phi:E\to E$ be a self-adjoint odd symmetric bundle map. Then the operator 
\[
    B_\Phi\ := \ D_V^+\ + \ i\Phi
\]
is odd symmetric. To connect it to the graded Callias-type operator of Section~\ref{SS:Calliasoperators}, we define
\[
  F\ = \ \begin{pmatrix}
        0&F^-\\F^+&0
    \end{pmatrix} \ := \ 
    \begin{pmatrix}
        0&-i\Phi \\i\Phi&0
    \end{pmatrix}
\]
and set 
\[
    \B_F\ := \  D_V\ + \ F \ = \  \begin{pmatrix}
        0&B_\Phi^*\\ B_\Phi&0
    \end{pmatrix}\ = \ 
    \begin{pmatrix}
        0&D^-_V - i\Phi\\D^+_V+ i\Phi&0
    \end{pmatrix}.
\]
Then $\B_F$ is a graded odd symmetric operator acting on $\Gamma(M,E\oplus E)$.  As usual, we set $\B_F^+:= B_\Phi$,  $\B_F^-:= B_\Phi^*$.

%---------------------------
\begin{definition}\label{D:ungradedCallias}
We say that $B_\Phi= D^+_V+i\Phi$ is an {\em ungraded Callias-type operator of Dirac type} if $\B_F$ is a Callias-type operator in the sense of Definition~\ref{D:Callias}.
\end{definition}
The admissibility conditions of Definition~\ref{D:admissible endomorphism} become, in this case, 
\begin{enumerate}
\item[(i)]  the restriction of the commutator $[D_V^+,\Phi] \ := \ D_V^+\, \Phi\ - \ \Phi\, D_V^+$ to $M\setminus K$ is an operator of order $0$;
\item[(ii)] there exist a constant $c>0$ such that
\begin{equation}\label{E:ungraded admissible endomorphism}
    \Phi(x)^2\ \ge \ \big\|[D_V^+,\Phi](x)\big\| \ +\ c, \qquad
    \text{for all}\quad x\in M\setminus K.
\end{equation}
\end{enumerate}

The compact set $K$ is called an \textit{essential support of $\Phi$}.

Note that Condition (i)  is equivalent to
\begin{equation}\label{E:[Phi,c]}
	[c(\xi),\Phi]\ = \ 0, \qquad\text{for all}\quad \xi\in T_xM,\ x\in  M\setminus K.
\end{equation}

In this section, we study the $\tau$-index of the operator $B_\Phi=B_F^+$. 

%--------------
\subsection{The $\ZZ_2$-valued Callias-type theorem}\label{SS:Calliastheorem}
Suppose that there is a \textit{$\theta$-invariant  partition} $M=M_-\cup_N M_+$, where $N=M_-\cap M_+$ is a smooth compact hypersurface and $M_-$ is a \textbf{compact} submanifold, whose interior contains an essential support $K$ of $F$, cf. Figure~\ref{fig:Callias theorem}.  The $\theta$-invariance of the partition meas that $M_\pm$ are $\theta$-invariant submanifolds of $M$. Then $N$
 is also $\theta$-invariant. 
\begin{figure}[ht]
    \centering
    \vskip-.5cm
    \includegraphics[width=0.50\textwidth]{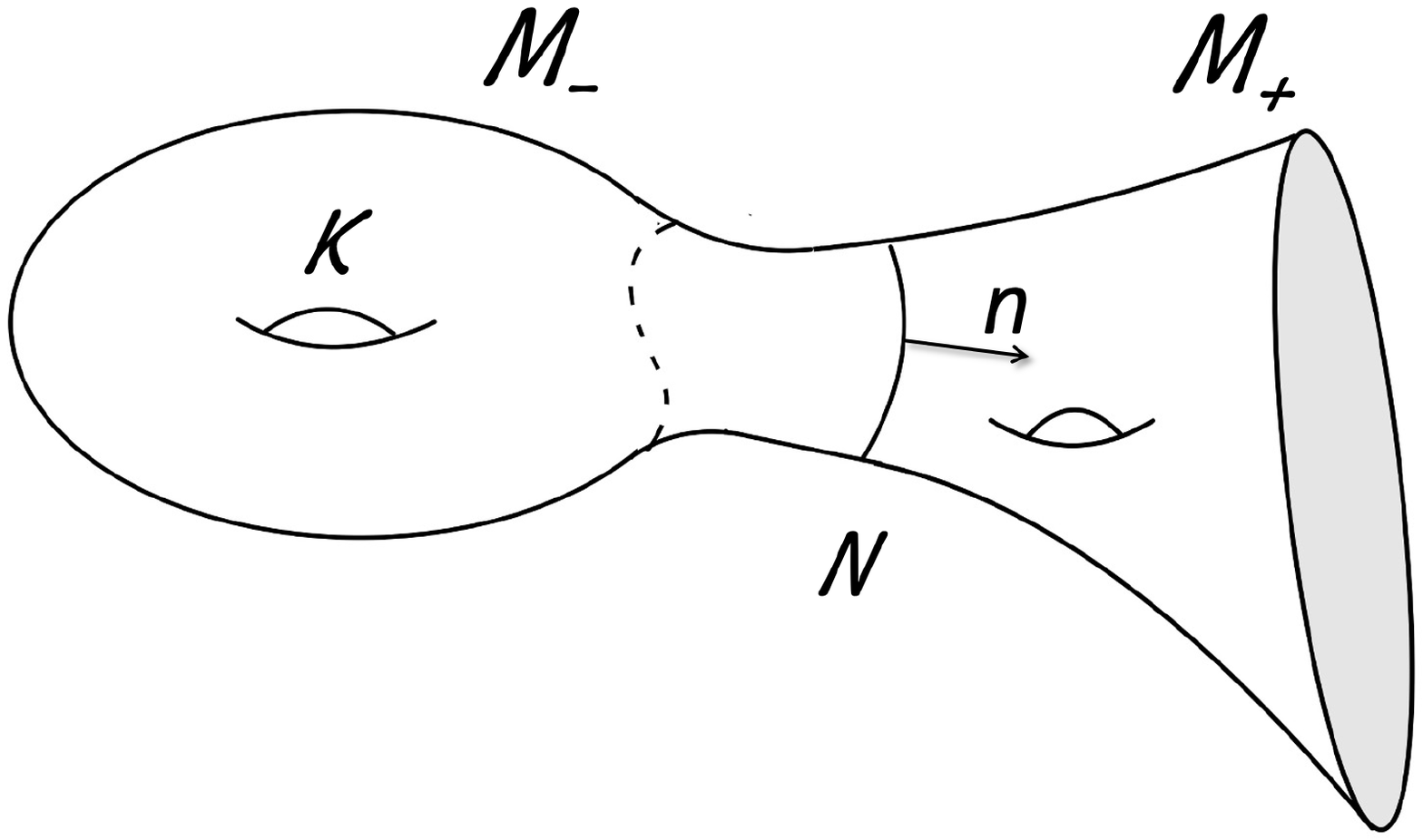}
    \vskip-1.5cm\caption{A $\theta$-invariant partition of $M$}
    \label{fig:Callias theorem}
\end{figure}

We denote by $n$ the unit normal vector field to $N$ pointing in the direction of $M_+$, cf. Figure~\ref{fig:Callias theorem}. Since $M_-$ is compact and $M_+$ is not compact, $\theta:M_\pm\to M_\pm$. It follows that $\theta_*n= n$. Hence, \eqref{E:c-tau} implies that  $c(n)$ 
\begin{equation*}
    \theta^E \, c\big(n(x)\big)\,(\theta^E)^{-1}\ = \ c\big(n(\theta x)\big).
\end{equation*}
or, equivalently, 
\begin{equation}\label{E:tau n}
    \tau\, c(n)\, \tau^{-1}\ = \ c(n).
\end{equation}

Let $E_N$ be the  restriction of $E$ to $N\subset M$.  Since $N$ is $\theta$-invariant submanifold, $E_N$ is a $\theta^E$-invariant and, hence, is a quaternionic bundle over $(N,\theta)$.

By Condition (ii) of Definition~\ref{D:admissible endomorphism}, zero is not in the spectrum of $\Phi(x)$ for all $x\in N$. Therefore we have a direct sum decomposition
\begin{equation}\label{E:decomposition of Sigma_N}
E_N=E_{N+}\oplus E_{N-},
\end{equation}
where the fiber of $E_{N+}$ (resp. $E_{N-}$) over $x\in N$ is the image of the spectral projection of $\Phi(x)$ corresponding to the interval $(0,\infty)$ (resp. $(-\infty,0)$).
By \eqref{E:[Phi,c]} the endomorphism $\Phi$ commutes with the Clifford multiplication. Hence $c(\xi):E_{N\pm}\to E_{N\pm}$ for all $\xi\in TM$. It follows that both bundles, $E_{N+}$ and $E_{N-}$, inherit  the Clifford action of $TM$. Also, since $\Phi$ commutes with $\tau$, both bundles $E_{N+}$ and $E_{N-}$ are quaternionic bundles over $N$.

Recall that we denote by $n$ the unit normal vector field pointing at the direction of $M_+$. The operator $c(n)$ defines an endomorphism $\gamma:E_{N\pm}\to E_{N\pm}$.
Since $\gamma^2=-1$, the endomorphism $\alpha:=-i \gamma$  is an involution and induces a grading 
\begin{equation}\label{E:grading on SigmaN}
	E_{N\pm}\ = \ E_{N\pm}^+\oplus E_{N\pm}^-,
\end{equation}
where $E_{N\pm}^\pm\subset E_{N\pm}$ is the span of the eigenvectors of $\alpha$ with eigenvalues $\pm1$. The Clifford action of $TN$ on $E_{N\pm}$ is graded with respect to this grading, i.e. $c(\xi):E_{N\pm}^\pm\to E_{N\pm}^\mp$ for all $\xi\in TN$. By \eqref{E:tau n}, $\tau$ is odd with respect to this grading (recall that $\tau$ is anti-linear and, hence, anti-commutes with $\alpha= -ic(n)$). Hence, $E_{N\pm}$ are graded quaternionic vector bundles over $N$.

Let $\nabla^{E_N}$ be the connection on $E_N$ obtained by restricting the connection on $E$. It does not, in general, preserve decomposition~\eqref{E:decomposition of Sigma_N}. 
We define a connection $\nabla^{E_{N\pm}}$ on the bundle $E_{N\pm}$ by 
\begin{equation}\label{E:nSigmaN+}
	\nabla^{E_{N\pm}}_\xi s^\pm\ =\ 
	\operatorname{pr}_{E_{N\pm}}\left(\nabla^{E_N}_\xi s^\pm\right),
	\qquad s^\pm\in C^\infty(N,E_{N\pm}),\ \xi\in TN,
\end{equation}
where $\operatorname{pr}_{E_{N\pm}}$ is the projection onto the bundle $E_{N\pm}$. Since $\gamma$ commutes with both, the connection $\nabla^{E_N}$ and the projection $\operatorname{pr}_{E_{N\pm}}$, it also commutes with the connection $\nabla^{E_{N\pm}}$.
By \cite[Lemma~2.7]{Anghel90},  it is a Hermitian Clifford connection. We denote by  $\dirac_{N}:= \left(\begin{smallmatrix}
0&\dirac_{N-}\\ \dirac_{N+}&0\end{smallmatrix}\right)$ the Dirac operator on $N$ associated with these connections:
\begin{equation}\label{E:DN}
	\dirac_{N\pm}\ := \ \sum_{i=1}^{n-1} c(e_i)\,\nabla^{E_{N\pm}}_{e_i}.
\end{equation}
where $(e_1,\ldots,e_{n-1})$ is an orthonormal frame of $TN$. The operators $\dirac_{N\pm}$ are odd symmetric and are odd  with respect to the grading \eqref{E:grading on SigmaN}, i.e. they have the form 
\[
	\dirac_{N\pm}\ = \ \begin{pmatrix}
	0&\dirac_{N\pm}^-\\\dirac_{N\pm}^+&0
	\end{pmatrix},
\]
where $\dirac_{N\pm}^+$ (respectively $\dirac_{N\pm}^-$) is the restriction of $\dirac_{N\pm}$ to $E_{N\pm}^+$ (respectively $E_{N\pm}^-$).

%\begin{equation}\label{E:indexDNpm}
%	\ind_\tau \dirac_{N\pm}\ = \ \dim_\tau\ker \dirac_{N\pm}^+. 
%\end{equation}
The next theorem is a $\ZZ_2$-valued analog of the Callias index theorem,  
 \cite[Theorem~1.5]{Anghel93} and \cite[Theorem~2.9]{Bunke95}.

%------------------------------------
\begin{theorem}[\textbf{Callias-type theorem for odd symmetric operators}]\label{T:computation of odd-dimensional case}
\begin{equation}\label{E:Callias-type theorem for odd symmetric operators}
	\ind_\tau B_\Phi\ =\ \ind_\tau \dirac_{N+}^+.
\end{equation}
\end{theorem}

We prove the theorem in the next section. We finish this section with a couple of immediate corollaries.

%-----------------
\begin{corollary}\label{C:Phi=-Phi}
$\ind_\tau{}B_\Phi= \ind_\tau{}\dirac_{N-}^+$. Hence, 
\begin{equation}\label{E:N+=-N-}
	\ind_\tau{}\dirac_{N+}^+\ = \ \ind_\tau{}\dirac_{N-}^-.
\end{equation}
\end{corollary}
%---
\begin{proof}[Proof of Corollary \ref{C:Phi=-Phi}]
Applying Theorem~\ref{T:computation of odd-dimensional case} to the operator $B_{-\Phi}$ we get 
\begin{equation}\label{E:ind-Phi}
    \ind_\tau B_{-\Phi}\ = \ \ind_\tau{}\dirac_{N-}^+.
\end{equation}
Since $B_\Phi$ is odd symmetric we have 
\[
    B_{-\Phi}\ =\ D_V^+-i\Phi \ =\  (D_V+i\Phi)^*
    \ = \ B_\Phi^* \ = \ 
    \tau\,B_\Phi\,\tau^{-1}.
\]
Hence, $\ind_\tau B_\Phi= \ind_\tau B_{-\Phi}$ and the Corollary follows from \eqref{E:ind-Phi}. 
\end{proof}

%------------
\begin{remark}\label{R:Phi=-Phi}
At first glance, the bundles $E_{N+}$ and $E_{N-}$ are not related. However, both operators, $\dirac_{N+}$ and $\dirac_{N-}$, are induced by the same operator $D$. This implies that the direct sum $\dirac_{N+}\oplus{}\dirac_{N-}$ is {\em cobordant to 0} in the sense of \cite{Br-cob}, and the cobordism is given by the operator $D_V$. In fact, this equality \eqref{E:N+=-N-} is equivalent to the cobordism invariance of the index.
\end{remark}

If $\Phi$ is an admissible endomorphism of $E$ then so is $\lambda{}\Phi$ for all $\lambda>1$. As an immediate corollary of Theorem~\ref{T:computation of odd-dimensional case} we obtain the following

%----
\begin{corollary}\label{C:lambdaPhi}
The index $\ind_\tau{}B_{\lambda \Phi}$ is independent of $\lambda\ge 1$. 
\end{corollary} 
%--
\begin{remark}\label{R:lambdaPhi}
If the endomorphism $\Phi$ is bounded then the domain of the operator $B_{\lambda \Phi}$ is independent of $\lambda$ and the corollary follows directly from the stability of the $\tau$-index. However, if $\Phi$ is not bounded the domain of $B_{\lambda \Phi}$ depends on $\lambda$ and the corollary is not {\em a priori} obvious. 
\end{remark}

%--------------------------------
\section{The proof of Theorem~\ref{T:computation of odd-dimensional case}}\label{S:proof of computation of odd-dimensional case}

In this section, we prove Theorem~\ref{T:computation of odd-dimensional case}
%--------------
\subsection{The plan of the proof of Theorem~\ref{T:computation of odd-dimensional case}.}\label{SSplanprCallias}
The rest of this section is devoted to the proof of Theorem~\ref{T:computation of odd-dimensional case}. The proof is essentially the same as in the classical case, cf. \cites{Anghel93Callias,Bunke95}. We mostly follow the exposition in \cite{BrCecchini17} and carefully check that all construction can be made $\tau$-equivariantly. We add one new step in Section~\ref{SS:deformation on neck}, which makes the other steps drastically easier.  Before giving all the details, let us briefly sketch the main steps of the proof. 

\subsubsection*{Step 1}
First, we consider the case when $M=N\times\RR$ is a cylinder and assume that all the structures on $M$ are products of the structures on $N$ and $\RR$. We set $M_-:= N\times(-\infty,0]$ and $M_+:=N\times[0,\infty)$ This is not exactly the situation of 
 Theorem~\ref{T:computation of odd-dimensional case} since both $M_+$ and $M_-$ are not compact. However, we still can construct the operator $\dirac_N$ as in Section~\ref{SS:Calliastheorem}.  Using the separation of variables we explicitly compute the kernel of $B_\Phi$ and show that \eqref{E:Callias-type theorem for odd symmetric operators} holds in this case. 

\subsubsection*{Step 2} We consider a manifold $M= M_-\cup_N (N\times[-1,\infty))$ with a cylindrical end and assume that the restrictions of all the structures on $M$ to the cylindrical part $N\times[-1,\infty)$ are the same as in the previous step. We apply the Relative Index Theorem~\ref{T:relative index} to prove that the $\tau$-index of the Callias type operator in this case is equal to the $\tau$-index on the cylinder (Figure~\ref{fig:Cylindrical} gives a good idea of how it is done). Hence, Theorem~\ref{T:computation of odd-dimensional case} holds in this case by Step~1.

\subsubsection*{Step 3} We consider the general manifold $M$. We identify a small neighborhood $U(N)$ of the hypersurface $N$  with the product $N\times(-1,1)$ and deform all the structures in $U(N)$ to that considered in Step 1. The $\tau$-index does not change under this deformation. 

\subsubsection*{Step 4} Finally, we apply the Relative index theorem to reduce the general case to the case of a manifold with a cylindrical end (see Figure~\ref{fig:general case}). 

The rest of this section is occupied with the details of this proof.

%--------------
\subsection{Step 1. The case of a cylinder}\label{SS:model operator}
First, we consider the case when $M$ is a cylinder $N\times\RR$ and show that Theorem~\ref{T:computation of odd-dimensional case} holds in this case. 

Let $(N,\theta^N)$ be a {\em closed} involutive manifold of even dimension $n-1$. Let $(E_N=E^+_N\oplus E^-_N,\theta^{E_N})$ be a graded quaternionic Dirac bundle over $N$. Let $\nabla^{E_N}=\nabla^{E^+_N}\oplus\nabla^{E^-_N}$ denote the $\tau$-invariant Clifford  connection on $E_N$ and let $\dirac_N$
\[
	\dirac_N\ =\ \begin{pmatrix}0&\dirac_N^-
        \\ \dirac^+_N&0\end{pmatrix},
\]
be the Dirac operator associated with $\nabla^{E_N}$ cf. Section~\ref{SS:Dirac}. 

Define the involution on the cylinder $N\times\RR$ by $\theta(x,t):= (\theta^Nx,t)$.
Let $p:N\times\RR\rightarrow N$ be the projection onto the first factor and denote by $\widehat{E}_N$ the pull-back bundle $p^\ast E_N$. Then 
\begin{equation}\label{E:widehatEN}
	\widehat{E}_N\ = \ \widehat{E}_N^+\oplus \widehat{E}_N^-,
	\qquad\text{where}\quad \widehat{E}_N^\pm:= p^*E_N^\pm.
\end{equation}
We denote by $\n^{\widehat{E}_N},\ \widehat{c}_N,\ h^{\widehat{E}_N}$, ... the structures induced on  $N\times\RR$ by corresponding structures on $N$. More precisely, $h^{\widehat{E}_N}$ is just the pull-back of the Hermitian metric on $E_N$. The natural Clifford action $\widehat{c}_N$ on $\widehat{E}_N$  is given by:
\begin{equation}\label{E:clifford action}
	\widehat{c}(\xi,t)
	\ = \
	c(\xi)\ + \ \gamma t,\qquad  
	(\xi,t)\in T_{(x,r)} (N\oplus\RR)=
	    T_x N\oplus\RR,\qquad (x,r)\in N\times \RR,
\end{equation}
where $c$ is the Clifford action of $TN$ on $E_N$ and $\gamma\big|_{\widehat{E}_N^\pm}=\pm i$. Note that this action does not preserve the grading \eqref{E:widehatEN}. Endowed with the pull-back connection $\nabla^{\widehat{E}_N}$ induced by the connection on $E_N$, the bundle $\widehat{E}_N$ becomes an ungraded Dirac bundle. Let $\widehat{\dirac}_N$ denote the Dirac operator associated with the connection $\nabla^{\widehat{E}_N}$. With respect to the decomposition 
\begin{equation}\label{E:decomposition of sections over the cylinder}
L^2(N\times\RR,\widehat{E})=L^2(E_N)\otimes L^2(\RR).
\end{equation}
$\widehat{\dirac}_N$ has the form
\begin{equation}\label{E:hatD}
	\widehat{\dirac}_N\ :=\ \dirac_N\otimes1\ +\ \gamma\otimes\partial_r,
\end{equation}
where $\partial_r$ denotes the operator of the derivation in the axial direction of the cylinder $N\times \RR$.

The quaternionic structure $\theta^{E_N}$ on $E_N$ extends naturally to a quaternionic structure $\theta^{\widehat{E}_N}$ on $\widehat{E}_N$ and $\widehat{\dirac}_N$ is odd symmetric with respect to this structure. 

Let $h:\RR\rightarrow \mathbb{R}$ be a smooth function such that
\begin{equation}\label{E:function h}
	h(r)=\left\{
	  \begin{array}{rr}-1,\qquad &r<-1/2\\1\,,\qquad &r>1/2,
	  \end{array}
	 \right.
\end{equation}
By a slight abuse of notation, we will denote by 
$h$ also the induced function $h:N\times\RR\to[-1,1]$. Then $h$ is a $\tau$-invariant function on $N\times\RR$. The multiplication by  $h$ is an admissible odd symmetric endomorphism of the ungraded Dirac bundle $\widehat{E}_N$ (see Section~\ref{SS:CalliasDirac}).

We define the \emph{model operator} ${\bf M}: \Gamma(N\times\RR,\widehat{E}_N)\to \Gamma(N\times\RR,\widehat{E}_N)$ 
\begin{equation}\label{E:twisted Callias type operator}
	{\bf M}\ :=\ \left(\begin{array}{cc}0&{\bf M}_-\\{\bf M}_+&0\end{array}\right),
\end{equation}
where ${\bf M}_\pm:=\widehat{\dirac}_N\pm ih$. This is an odd symmetric Callias-type operator (indeed, $\tau\,(\widehat\dirac_N+i)\tau^{-1}= \widehat\dirac_N-i$ because $\tau$ is anti-linear). It is shown in \cite[Corollary~4.10]{BrCecchini17} that 
\begin{equation}\label{E:M=D}
    \dim\ker \mathbf{M}_\pm \ = \ \dim \ker\dirac_N^\pm.
\end{equation}
This means that Theorem~\ref{T:computation of odd-dimensional case} holds for the case of a cylinder. 

%--------------
\subsection{Step 2. A manifold with a cylindrical end}\label{SS:cylindrical end}
Assume that $M_+= N\times[-1,\infty)$ for some compact manifold $N$. In other words, we assume that 
\begin{equation}\label{E:cylindrical end}
	M\ = \ M_-\cup_N \left(N\times[-1,\infty)\right),
\end{equation}
where  $M_-$ is a compact manifold with boundary, whose interior contains an essential support of the potential $\Phi$. 

As in \eqref{E:decomposition of Sigma_N}, the restriction $E_N$ of $E$ to $N\times\{0\}$ decomposes into a direct sum of $\tau$-invariant sub-bundles: $E_N= E_{N+}\oplus E_{N-}$. Let $\widehat{E}_{N\pm}$ be the corresponding quaternionic Dirac bundles on $N\times\RR$ defined in Section~\ref{SS:model operator}. 
We assume that there is a fixed $\tau$-invariant Dirac bundles isomorphism between the restriction of $E$ and of $\widehat{E}_{N_+}\oplus \widehat{E}_{N_-}$ to $N\times[-1,\infty)$. This means that the restrictions of the metrics, the connection, the Clifford action, and the potential $\Phi$ to $N\times[-1,\infty)$ are equal to the corresponding structures $\n^{\widehat{E}_{N+}\oplus\widehat{E}_{N_-}},\ \widehat{c}_N,\ h^{\widehat{E}_{N+}\oplus\widehat{E}_{N_-}}$ on the cylinder (cf. Section~\ref{SS:model operator}). We also assume that the restriction of potential $V$ of the operator \eqref{E:D+V Callias} to $N\times[-1,\infty)$ vanishes.  From now on we will identify the restriction of $E$ and $\widehat{E}_{N_+}\oplus \widehat{E}_{N_-}$ to the cylinder and write 
\begin{equation}\label{E:E=EN+=EN-}
    E|_{N\times[-1,\infty)}= \widehat{E}_{N_+}\oplus \widehat{E}_{N_-}.
\end{equation} 
Further, we assume that the restriction of the potential $\Phi$ to $N\times[-1,\infty)$ is equal to the grading operator corresponding to the decomposition \eqref{E:E=EN+=EN-}, i.e. that $\Phi|_{\widehat{E}_{N\pm}}= \pm1$. Finally, we denote by $\theta_N$ and $\theta^E_N$ the restrictions of $\theta$ and $\theta^E$ to $N$ (recall that $N$ is a $\theta$ invariant submanifold) and  assume that the restrictions $\widehat{\theta}_N,\ \widehat{\theta}^E_N$ of $\theta$ and $\theta^E$ to $N\times\RR$ are equal to $\theta_N\times1$  and $\theta_N^E\times1$ respectively.

\begin{lemma}\label{L:cylindrical end}
Under the above assumptions, 
\begin{equation}
	\ind_\tau B_\Phi\ =\ \ind_\tau \dirac_{N+}^+.
\end{equation}
\end{lemma}
%----
\begin{proof}
Let $-M$ be a copy of $M$ with the opposite orientation. We denote by $-E$ the bundle $E$ viewed as a vector bundle over $-M$, endowed with the {\em opposite}  Clifford action. This means that a vector $\xi\in TM$ acts on $-E$ by $c(-\xi)$  (for more details about this construction we refer to \cite[Section~2.3.2]{Bunke95} and \cite[Chapter~9]{BoosWoj93book}). We need the change of the Clifford action later to glue $-E$ and $E$ together. 

There is a natural orientation preserving isometry  
\[
	\psi:-M\ \overset{\sim}\longrightarrow \  (N\times(-\infty,1])\cup_N(-M_-),
\]
such that 
\[\begin{aligned}
	\psi(x)\ =\ x, \qquad &\text{if}\quad x\in -M_-\simeq M_-\\ 
 	\psi(y,t)\ =\ (y,2-t), \qquad &\text{if}\quad (y,t)\in N\times[-1,\infty).
\end{aligned}\]
To simplify the notation we will skip $\psi$ from the notation and simply write 
\begin{equation}\label{E:-W}
	-M\ =\  (N\times(-\infty,1])\cup_N(-M_-).
\end{equation}

Let $\Gamma_N: \widehat{E}_{N_+}\oplus \widehat{E}_{N_-}\to \widehat{E}_{N_+}\oplus \widehat{E}_{N_-}$ be the operator defined in \eqref{E:Gamma}. We denote by the same symbol $\Gamma_N$ the induced operators on $\widehat{E}_{N_+}\oplus \widehat{E}_{N_-}$ and $-E|_{{M\times((-\infty,1]}}$. Then $\Gamma_N$ anti-commutes with $c(\xi)$ for $\xi$ tangent to $N$ and commutes with $c\big(\frac{\p}{\p t}\big)$. Hence, it defines an isomorphism of quaternionic Clifford bundles
\begin{equation}\label{E:-E restricted to cylinder}
	\Gamma_N:\,-E\big|_{N\times(-\infty,1]}\ \simeq\ 
	\widehat{E}_{N_+}\oplus \widehat{E}_{N_-}\big|_{N\times(-\infty,1]}.
\end{equation}

 Let $\Phi_1$ be the potential on $-E$ whose restriction to $N\times((-\infty,1]$ is equal to
\[
    \Phi_1\ = \ 
    \begin{pmatrix}
        h&0\\0&-1
    \end{pmatrix}
\]
with respect to decomposition \eqref{E:E=EN+=EN-} and which is equal to $\Phi$ on $-M_-\simeq M_-$. Here $h$ is given by \eqref{E:function h}. This is a smooth self-adjoint $\tau$-invariant bundle map. For $t>1/2$ this map is the constant matrix $\left(\begin{smallmatrix}1&0\\0&-1\end{smallmatrix}\right)$.

We identify the restriction of  both bundles $E$ and $-E$ to $N\times[-1,1]$  with the restriction of $\widehat{E}_N$ to  $N\times[-1,1]$. So we have a fixed isomorphism between them. Under this isomorphism the restrictions of $\Phi$ and $\Phi_1$  to $[1/2,1]$ are equal. Hence, we can apply the relative index theorem with $M_0:=M$, $M_1:= -M$,  $U(N):= N\times(1/2,1)$, and we cut along the surface $N\times\{3/4\}\simeq N$.  Then, $M_2\simeq M_-\cup_N \big(N\times[-1,1]\big)\cup_N (-M_-)$ and $M_3\simeq N\times \RR$, cf. Figure~\ref{fig:Cylindrical}.

\begin{figure}[ht]
    \centering
    \vskip-2cm
    \includegraphics[width=1\textwidth]{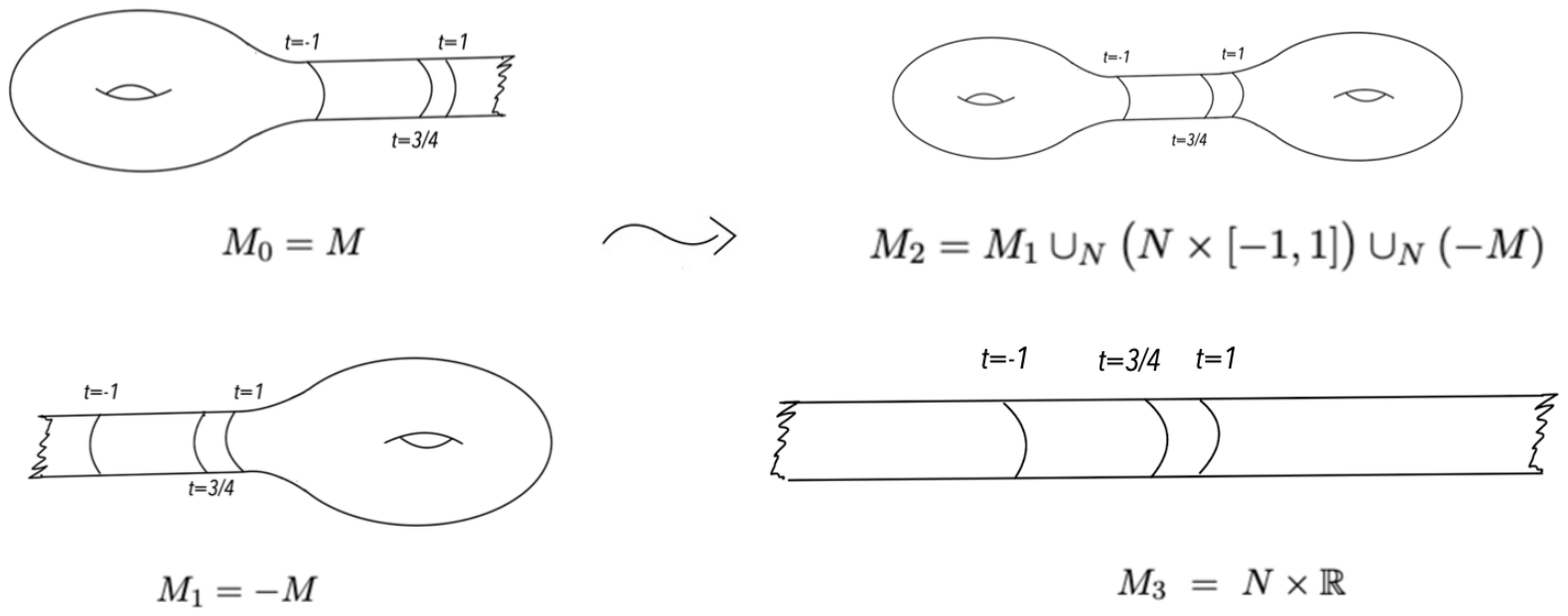}
    \vskip-2.5cm\caption{Cutting and pasting a manifold with a cylindrical end.}
    \label{fig:Cylindrical}
\end{figure}
Let $B_0,\ldots,B_3$ be the ungraded Callias-type operators on these manifolds. 

Consider a new potential $\widetilde{\Phi}_1:= -\ID$ on $M_1$ and let $\widetilde{B}_1$ be the corresponding Callias-type operator. The restrictions of $\widetilde{\Phi}_1$ and $\Phi_1$ to the part of the cylinder with $t<-1/2$ coincide. Hence the difference $\widetilde{\Phi}_1-\Phi_1$ has compact support and $\ind_\tau \widetilde{B}_1= \ind_\tau B_1$ by Lemma~\ref{L:compactperturbation}. Since the essential support of $\widetilde{\Phi}_1$ is empty, $\ind_\tau \widetilde{B}_1=0$ by Lemma~\ref{L:essential support}. Thus 
\begin{equation}\label{E:indB1}
        \ind_\tau B_1\ = \ 0.
\end{equation}

Since $M_2$ is a compact manifold of odd dimension, 
\begin{equation}\label{E:indB2}
        \ind_\tau B_2\ = \ 0,
\end{equation}
by Lemma~\ref{L:vanishing on compact}. From the relative index Theorem~\ref{T:relative index}, \eqref{E:indB1}, and \eqref{E:indB2} we conclude that 
\begin{equation}\label{E:indBF=indB3}
       \ind_\tau B_F\ := \ 
       \ind_\tau B_0\ = \ \ind_\tau B_3.
\end{equation}
The operator $B_3$ on the cylinder $M_3\simeq N\times \RR$ is a direct sum of operators $B_3|_{\widehat{E}_{N_+}}$ and $B_3|_{\widehat{E}_{N_-}}$.
\[
    \Phi_3\ = \ \begin{pmatrix}
        h&0\\0&-1
    \end{pmatrix}
\]
for all $t\not\in (-1,1)$. Hence, the restriction of  $\Phi_3$ to $\widehat{E}_{N_-}$ is -1. Thus the essential support of this restriction is empty and 
\begin{equation}\label{E:B3EN-=0}
    \ind_\tau B_3^+|_{\widehat{E}_{N_-}}\ = \ 0
\end{equation}
by Lemma~\ref{L:essential support}.

The restriction of $\Phi_3$ to $\widehat{E}_{N_+}$ coincides with $h$ outside of a compact set $N\times(-1,1)$. Thus, by Lemma~\ref{L:compactperturbation},  
\begin{equation}\label{E:B3EN+=inddirac}
    \ind_\tau B_3|_{\widehat{E}_{N_+}}\ = \ \ind_\tau \textbf{M}_+\ = \ \ind_\tau \dirac_{N+}^+,
\end{equation}
where in the last equality we used \eqref{E:M=D}.

Combining \eqref{E:indBF=indB3}, \eqref{E:B3EN-=0}, and \eqref{E:B3EN+=inddirac}, we obtain \eqref{E:cylindrical end}.
\end{proof}

%--------------
\subsection{Step 3. Deformation of the operator on a neighborhood of $N$}\label{SS:deformation on neck}
Let us fix a $\tau$-equivariant identification between  a neighborhood $U(N)$ of $N$ in $M$ and  the product $U\simeq N\times(-1,1)$.   In this step, we deform all the structures on $U(N)$ to product structures so that the index of $B_\Phi$ does not change under this deformation. Then, it is enough to prove the theorem for the case when all the structures are products near $N$. We prove this in the next Subsection using a cutting and pasting method similar to the one used in Subsection~\ref{SS:cylindrical end}.

Let $\alpha_s:[-1,1]\to [-1,1]$ ($s\in[0,1]$)be a smooth family of smooth maps (by this we mean that the map $(s,t)\mapsto \alpha_s(t)$ is smooth) such that 
\begin{enumerate}
    \item $\alpha_0(t)=t$ for all $t\in [-1,1]$;
    \item $\alpha_s(1)= 1$ and $\alpha_s(-1)=-1$ for all $s\in [0,1]$;
    \item $\alpha_1(t)=0$ for $t\in [-1/2,1/2]$.
\end{enumerate}

Let $\beta_s:M\to M$ be the map whose restriction to $M\setminus U(N)$ is equal to the identity map and whose restriction to $U(N)\simeq N\times(-1,1)$ is given by  $\beta_s(x,t):= \big(x,\alpha_s(t)\big)$. Then $\beta_s$ is a family of smooth $\tau$-equivariant maps. 

We denote by $g^{TM}_s, h^E_s, \n^E_s,c_s(\xi),\Phi_s, ...$ the pull-backs of corresponding structures by the map $\beta_s$.  When $s=1$ these structures are product on $N\times[-1/2,1/2]\subset U(N)$. Let $B_{\Phi,s}$ be the Callias-type operator corresponding to these structures. Then it is an odd symmetric self-adjoint operator. Since the family of operators $B_{\Phi,s}$ ($s\in[0,1]$) is constant outside of the compact set $U(N)$, all these operators have the same domain. Hence, $B_{\Phi,s}$ is a continuous family of bounded operators from the Sobolev space $W^{1,2}(M,E\oplus E)$ to the space of square-integrable sections. Hence, by the stability of the index (Theorem~\ref{T:homotopyindex}),
$\ind_\tau B_{\Phi,s}^+$ is independent of $s$. Hence, it is enough to prove Theorem~\ref{T:computation of odd-dimensional case} for the operator $B_{\Phi,1}$, which corresponds to product structures near $N$.

%--------------
\subsection{Step 4. Proof of Theorem~\ref{T:computation of odd-dimensional case} in the general case} \label{SS:general case}
We now assume that all the structures are product near $N$ as in Section~\ref{SS:deformation on neck}.  Let $\chi:\RR\to [0,1]$ be a smooth function such that 
\[
    \chi(t) \ = \ 
    \begin{cases}
        1 \quad&\text{for}\quad |t|\ge 1/2;\\
        0 \quad&\text{for}\quad |t|\le 1/3.
    \end{cases}
\]
Define a function $r:M\to [0,1]$ by  
\[
    \begin{cases}
        r(x)=1 \quad&\text{for}\quad x\in M\setminus U(N);\\
        r(y,t)=\chi(t) 
        \quad&\text{for}\quad (y,t)\in N\times(-1,1)\simeq U(N).
    \end{cases}
\] 
Then $r$ is a smooth  $\tau$-invariant function.

Let $\Phi(x)= U(x)|\Phi(x)|$  be the polar decomposition of the self-adjoint map $\Phi$(x).  Here $U(x)$ is a unitary  operator and $|\Phi(x)|$ is a non-negative operator.  The operator $|\Phi|$ is odd symmetric. Define a new potential 
\[
    \Phi'(x)\ := \ 
    \Phi(x)\, \Big(\,r(x)+\big(1-r(x)\big)\,|\Phi(x)|,\Big)^{-1}.
\]
Then $\Phi'(x)$ is equal to $\Phi(x)$ for $x\in M\setminus U(N)$. It follows that the essential support of $\Phi'$ is contained in $M_-$. Let $B_{\Phi'}$ denote the Callias-type operator corresponding to $\Phi'$. Then  
\begin{equation}\label{E:Phi=Phi'}
    \ind_\tau B_{\Phi'} \ = \ \ind_\tau B_\Phi
\end{equation}
by Lemma~\ref{L:essential support}.

The restriction of $\Phi'$ to $N\times[-1/3,1/3]\subset U(N)$  is equal to the grading operator corresponding to the decomposition \eqref{E:E=EN+=EN-}, i.e. to the potential on the manifold $M_-\cup_N \left(N\times[-1,\infty)\right)$ from Section~\ref{SS:cylindrical end}. 

Let $M_1:= N\times\RR$ be the cylinder. Consider the Dirac bundle $\widehat{E}_N= \widehat{E}_{N+}\oplus \widehat{E}_{N-}$ over $M_1$, cf. Section~\ref{SS:cylindrical end}.  Let $\Phi_1$ be the grading operator on $\widehat{E}_N$. Then all the data on $M$ and on $M_1$ coincide on $N\times[-1/3,1/3]$. Hence, we can apply the relative index theorem~\ref{T:relative index} to the pair $M_0:=M$ and $M_1$ where we cut along the manifold $N=N\times\{0\}$.  Then $M_2= M_-\cup_N \big(N\times[0,\infty)\big)$ and $M_3= \big(N\times(-\infty,0]\big)\cup_N M_+$, cf. Figure~\ref{fig:general case}.

\

\begin{figure}[ht]
    \centering \vskip-1cm
    \includegraphics[width=0.8\textwidth]{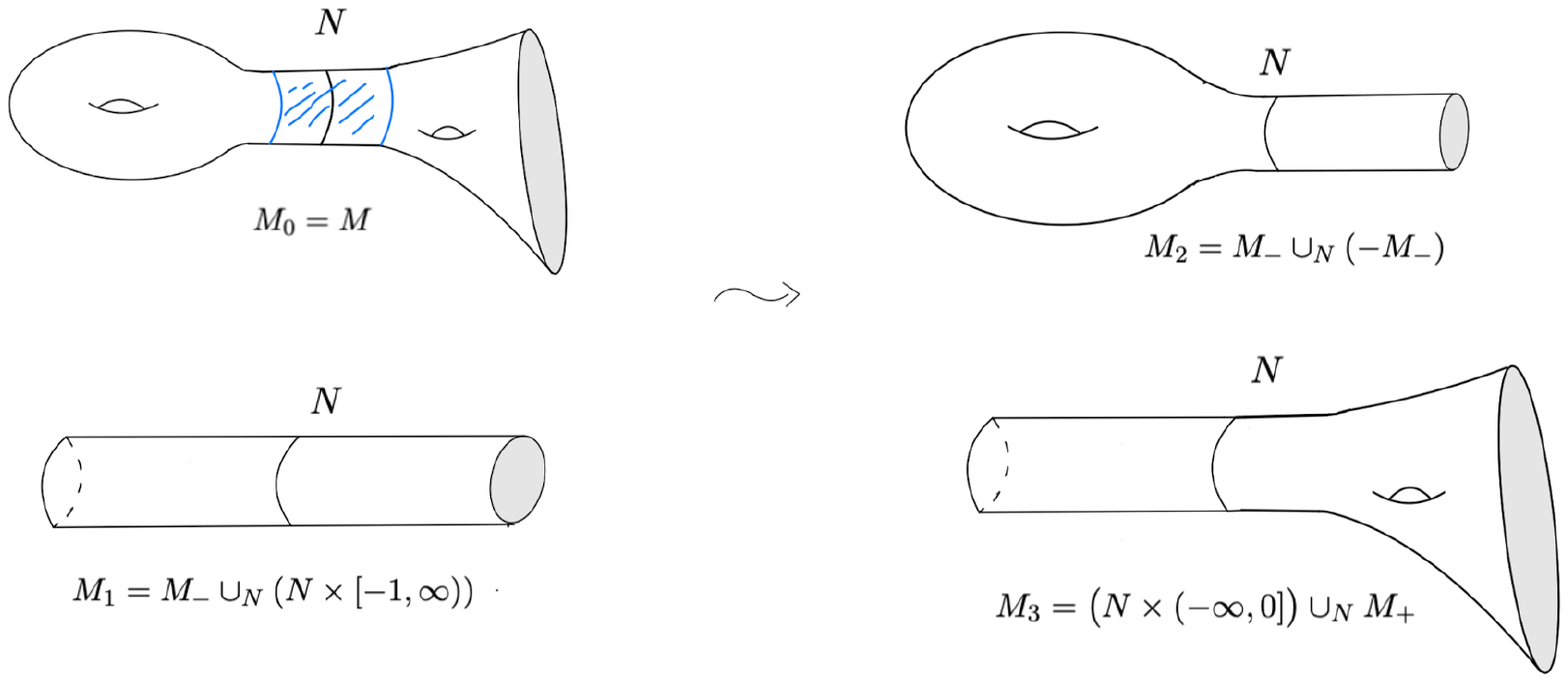}
    \vskip-2cm\caption{Cutting and pasting in the general case.}
    \label{fig:general case}
\end{figure}

Let $B_1$ denote the Callias-type operator on $M_1=N\times\RR$ corresponding to the potential $\Phi_1$. Since the essential support of $B_1$ is empty, $\ind_\tau B_1=0$ by Lemma~\ref{L:essential support}. The potential $\Phi_3$ which we obtain on $M_3= \big(N\times(-\infty,0]\big)\cup_N M_+$ also has empty essential support and the $\tau$-index of the corresponding Callias-type operator $B_3$ vanishes. Thus the Relative Index Theorem~\ref{T:relative index} implies that 
\[
    \ind_\tau B_{\Phi'}\ = \ \ind_\tau B_2,
\]
where $B_2$ is the Callias-type operator on $M_2= M_-\cup_N \big(N\times[0,\infty)\big)$. This is exactly the operator we studied in Section~\ref{SS:cylindrical end}. Theorem~\ref{T:computation of odd-dimensional case} follows now from Lemma~\ref{L:cylindrical end}.
\hfill$\square$

%------------------------------------------------------
%------------------------------------------------------
\section{Cobordism invariance of the $\tau$-index}\label{S:cobordism}

In this section, we show that the Callias-type Theorem~\ref{T:computation of odd-dimensional case} implies that the $\tau$-index of a Dirac-type operator on a compact even-dimensional manifold is preserved by a cobordism. To simplify the notation we only consider the case of {\em null-cobordism}, i.e. the cobordism between a manifold and the empty set. We prove that the $\tau$-index of a null-cobordant operator vanishes. A standard argument, cf., for example, \cite[Remark~2.10]{BrShi16}, shows that this is equivalent to the cobordism invariance of the index. 

A lot of formulations and constructions of this section are borrowed from \cite{BrShi16} and \cite[\S7]{BrCecchini17}.

%---------------------------------------------
\subsection{The settings}\label{SS:cobordismsettings}
Let $(M,\theta)$ be a closed involutive Riemannian manifold. Let $(E=E^+\oplus E^-,\theta^E)$ be a quaternionic Dirac bundle over $M$, cf. Section~\ref{SS:Dirac}. Recall that we assume that $\theta^E$ and $c(\xi)$ ($\xi\in TM$) are odd operators with respect to the grading. 

Let $\dirac$ be the corresponding generalized Dirac operator \eqref{E:Dirac}. Let $V=\left(\begin{smallmatrix}
    0&V^-\\V^+&0
\end{smallmatrix}\right)$ be a graded potential and consider the Dirac-type operator $D_V$, cf. \eqref{E:Diractype}. We assume that both operators,  $\dirac$ and $V$, are odd symmetric. Then the Dirac-type operator 
\[
    D_V \ := 
    \begin{pmatrix}
        0&\dirac^-+V^-\\
        \dirac^++V^+&0
    \end{pmatrix},
\]
cf. Section~\ref{SS:oddpotential}).

\newcommand{\oE}{\overline{E}}
\newcommand{\oD}{\overline{D}}
\newcommand{\oV}{\overline{V}}
\newcommand{\oPhi}{\overline{\Phi}}
%-----------
\begin{definition}\label{D:cobordismSigma}

A  {\em quaternionic null-cobordism} of $D_V$ is a pair $(W,\oE)$, where 
\begin{enumerate}
\item $(W,\theta^W)$ is a compact involutive manifold with boundary $\p{W}\simeq M$ and the restriction of $\theta^W$ to $M$ coincides with $\theta^M$;
\item $(\oE,\theta^{\oE})$ is a quaternionic Dirac bundle over $W$ such that the corresponding generalized  Dirac operator $\overline{\dirac}$ on $\oE$ is odd symmetric. 
\item 
There is an open neighborhood $U$ of\/ $\p{W}$ and a $\tau$-invariant metric-preserving diffeomorphism
\begin{equation}\label{E:diffeo}
    \phi: M\times(-\epsilon,0] \ \to \ U.
\end{equation}
Here $\tau$ acts on the first factor of $M\times (-\epsilon,0]$.
\item 
Let $p:M\times \RR\to M$ denote the natural projection and let  $\widehat{E}= p^*E\simeq E\times\RR$ denote the lift of $E$ to $M\times\RR$. Let $t$ denote the coordinate on $\RR$. Then the bundle $\widehat{E}\oplus\widehat{E}$ has a natural structure of a quaternionic Dirac bundle over $M\times\RR$ with 
\[
    c\big(\frac{\p}{\p t}\big) \ = \ \begin{pmatrix}
        1&0\\0&-1
    \end{pmatrix}.
\]    
We assume that the restriction of $\oE$ to $M\times(-\epsilon,0]\subset U$ is $\tau$-equivariantly  isomorphic to the Dirac bundle $\widehat{E}$, defined in Section~\ref{SS:cylindrical end}.
\end{enumerate}  
\end{definition}

Let $\oV:\oE\to \oE$ be an odd symmetric potential, whose restriction to $U\simeq M\times (-\epsilon,0]$ is equal to the natural lift of $V$ to $E\times\RR$. Then the (ungraded) Dirac-type operator $\oD_{\oV}:= \overline{\dirac}+\oV$ is odd symmetric and its restriction to $U\simeq M\times (-\epsilon,0]$ has the form
\begin{equation}
    \widehat{D}\ := \ 
	\begin{pmatrix}
	i\frac{\partial}{\partial t}&D_V^-\\D_V^+&-i\frac{\partial}{\partial t}
	\end{pmatrix}.
\end{equation}

%------
\begin{theorem}\label{T:compact cobordism}
If $D$ is a qaternionically null-cobordant odd symmetric Dirac-type operator over an {\em even-dimensional} manifold $M$, then \(\ind_\tau{}D=0.\)
\end{theorem}

\newcommand{\tW}{\widetilde{W}}
\newcommand{\tSigma}{\widetilde{E}}
%-----
\begin{proof}
Consider the manifold $W':= W\cup_M(M\times[0,\infty))$. It carries a natural involution $\theta^{W'}$, whose restriction to $W$ is equal to $\theta^W$ and whose restriction to $M\times[0,\infty)$ is equal to $\theta^M\times1$.  

Let $E'$ denote the Dirac bundle over $W'$, whose restriction to $W$ is equal to $\oE$ and whose restriction to the cylinder $M\times[0,\infty)$ is equal to the bundle $\widehat{E}$. Let $V'$ be the potential, whose restriction to $W$ is equal to $\oV$ and whose restriction to  $M\times[0,\infty)$ is given by $V'(x,t)= V(x)$.  Let $D'$ be the Dirac operator whose restriction to $W$ is equal to $\oD$ and whose restriction to $M\times[0,\infty)$ is equal to  $\widehat{\dirac}+V'$ (recall that the operator $\widehat{\dirac}$ is defined in Section~\ref{SS:cylindrical end}). Then $D'$ is an ungraded Dirac-type operator, cf. Section~\ref{SS:CalliasDirac}.  

For an odd symmetric self-adjoint bundle map  $\Phi':E'\to E'$,  we denote by  $B'_{\Phi'}= D'+i\Phi$ the ungraded Callias-type operator on $W'$, cf. Section~\ref{SS:CalliasDirac}. 
Applying the Callias-type index theorem~\ref{T:computation of odd-dimensional case} to the operators $B'_{\ID}$ and $B'_{-\ID}$ we obtain
\[
	\ind_\tau B'_{\ID}\ = \ \ind_\tau D, \qquad 
	\ind_\tau B'_{-\ID}\ = \ 0.
\]
The proposition follows now from Corollary~\ref{C:Phi=-Phi}.
\end{proof}

%------------------------------------------------------
%------------------------------------------------------
\section{A $\ZZ_2$-valued analog of the Boutet de Monvel index theorem}\label{S:Boutet de Monvel}

Boutet de Monvel, \cite{BoutetdeMonvel78}, discovered an index theorem for a Toeplitz operator on a pseudo-convex domain. Guentner and Higson, \cite{GuentnerHigson96}, reinterpreted it in terms of the index of a Callias-type operator. Bunke, \cite{Bunke00}, generalized their result to Callias-type operators on general manifolds. In this section, we present a $\ZZ_2$-valued analog of the result of Bunke. The proof is practically the same, but easier, than in \cite{Bunke00}. We present it here for completeness.

%------------------
\subsection{The gap condition}\label{SS:gap condition}
Let $M$ be a complete even-dimensional Riemannian manifold, and let $E=E^+\oplus E^-$ be a graded Dirac bundle over $M$, cf. Section~\ref{SS:Dirac}. 
Let $\dirac$ be the associated generalized Dirac operator on $E$.   

We make the following fundamental 
\begin{assumption}[\textbf{Gap condition}]\label{A:gap}
Zero is an isolated point of the spectrum of $\dirac$.
\end{assumption}
%--------
\begin{remark}\label{R:pseudoconvex}
Let $M\subset \CC^n$ be a strongly pseudo-convex domain endowed with the Bergman metric,  cf. \cite[\S7]{Stein72book}. Then $M$ is a complete Riemannian manifold. Let $\dirac = \bar{\p}+\bar{\p}^*$ be the Dolbeault-Dirac operator on the space $\Omega^{n,*}(M,\CC^k)$ of $(n,*)$-forms on $M$ with values in the trivial bundle $\CC^k$. 

It follows from  \cite[\S5]{DonnellyFefferman83} that $\dirac$ satisfies the gap condition. This relates the results of this section to the original theorem of Boutet de Monvel. 
\end{remark}
Fix an integer $k$ and consider the graded  bundle 
\[
    E\otimes\CC^k\ = \ \big(E^+\otimes\CC^k\big)\oplus \big(E^-\otimes\CC^k\big).
\]
over $M$. The structure of a graded Dirac bundle on $E$ naturally extends to $E\times\CC^k$ and the corresponding generalized Dirac operator 
\[
    D\ := \ \dirac\otimes 1:\, \Gamma(M,E\otimes \CC^k)
    \ \to \ \Gamma(M,E\otimes \CC^k) 
\]
satisfies the gap condition~\ref{A:gap}.

%------------------
\subsection{The quaternionic structure}\label{SS:quaternionic_on_ExCk}
Let $\theta:M\to M$ be a metric preserving involution. Let 
\[
    \theta^E:\,E\ \to \ \theta^*E
\]
be an anti-linear map that preserves the Hermitian metric and such that $(\theta^E)^2=1$. Thus it is an anti-unitary involution. Note that, contrary to other parts of the paper, we consider an involution, not an anti-involution and we don't impose any conditions on its grading degree. In particular, in Example~\ref{E:example_for_Toeplitz} it will be even with respect to the grading. 

Let $\theta^{\CC^k}:\CC^k\to \CC^k$ be an anti-unitary anti-involution (By Lemma~\ref{L:even} it exists only if $k$ is even).  Consider an anti-unitary anti-involution $\theta^{E\otimes\CC^k}:= \theta^E\otimes\theta^{\CC^k}$. It makes $E\times\CC^k$  a graded quaternionic bundle. 

%------------------
\subsection{An example}\label{E:example_for_Toeplitz}
An important example of this situation appears as follows. Let $M$  be a pseudo-convex domain in $\CC^n$ and let $\theta:M\to M$ be an anti-holomorphic involution. As in Remark~\ref{R:pseudoconvex} we endow $M$ with the Bergman metric and consider the bundle $E= \Lambda^{n,*}T^*M$ of $(n,*)$-forms on $M$. We define an anti-linear \textit{involution} on $E$ by 
\[
    \theta^E:\,\omega\ \mapsto \ \overline{\theta^*\omega}.
\]
Here bar stands for the complex conjugation and $\theta^*$ is the pull-back of differential forms, which sends $(n,*)$-forms to $(*,n)$-forms (since $\theta$ is anti-linear). 

One readily checks that the Dolbeault-Dirac operator $D=\dirac\otimes1$ is odd-symmetric with respect to the  anti-unitary anti-involution $\theta^{E\otimes\CC^k}:= \theta^E\otimes\theta^{\CC^k}$.

If $M$ is a unit disc in $\CC$, this example is related to the Graf-Porta model for topological insulators with a times reversal symmetry \cite{GrafPorta13}. See \cite{BrSaeedi24a} for details. 

%------------------
\subsection{Matrix valued functions}\label{SS:matrix valued}
Let $Mat(k)$ denote the space of complex $k\times k$-matrices.  Let $BC(M,k)$ denote the Banach algebra of continuous bounded functions $f(x)$ on $M$ with values in $Mat(k)$. We define an anti-unitary anti-involution $\tau$ on $BC(M,k)$ by 
\begin{equation}\label{E:theta inv matrices}
    (\tau f)(x)\ = \ \theta^{\CC^k}\,f(\theta x)\, (\theta^{\CC^k})^{-1}.
\end{equation}
Let $BC(M,k)^\tau$ denote the set of $\tau$-invariant elements of $BC(M,k)$. Let $C_{0}(M,k)^{\tau}$ denote the closure of the subalgebra of functions with compact support in $BC(M,k)^{\tau}$. 

Let $C_g^\infty(M,k)^{\tau}\subset BC(M,k)^{\tau}$ denote the space of smooth $\tau$-invariant functions with values in $Mat(k)$ which are bounded and such that $df$ vanishes at infinity.  Then $C_0(M,k)^{\tau}\subset C_g(M,k)^{\tau}$. 

Let $C_g(M,k)^{\tau}$ denote the closure of $C_g^\infty(M,k)^{\tau}$ inside $BC(M,k)$. This is a commutative $C^*$-algebra. For $f\in C_g(M,k)^{\tau}$ we denote by $[f]$ its image in the quotient algebra $C_g(M,k)^{\tau}/C_0(M,k)^{\tau}$ and write $f\in [f]$.

%-----------------
\begin{definition}\label{D:invertible}
We say that $f$ is {\em invertible at infinity} if $[f]$ is invertible in $C_g(M,k)^{\tau}/C_0(M,k)^{\tau}$. Equivalently, this means that there exists a compact set $K\subset M$ and $C>0$  such that $f(x)$ is an invertible and $\|f(x)^{-1}\|<C$ for all $x\not\in K$.
\end{definition}

%------------------
\subsection{The Toeplitz operator}\label{SS:Toeplitz}
We now assume that the generalized Dirac operator 
\[
    D\ := \ \dirac\otimes 1:\, \Gamma(M,E\otimes \CC^k)
    \ \to \ \Gamma(M,E\otimes \CC^k). 
\]
is odd symmetric and that $\dirac$ and, hence, $D$ satisfy the gap condition, cf. Assumption~\ref{A:gap}.

Let $\calH:= \ker D \subset L^2(M,E\otimes\CC^k)$ and let $P: L^2(M,E\otimes\CC^k)\to \calH$ denote the orthogonal projection. 

For $f\in C_g(M,k)^{\tau}$ we denote by  $M_f:\Gamma(M,E\otimes \CC^k)\to \Gamma(M,E\otimes \CC^k)$ the operator of multiplication by $f$. Then, $M_f$ is an odd symmetric operator. We write  $M_f^\pm$ for the restriction of $M_f$ to $E^\pm$.

%----------
\begin{definition}\label{D:Toeplitz}
Let $f\in C_g(M,k)^{\tau}$, the Toeplitz operator is the operator
\[
    T_f\ := \ P\,M_f\,P:\,\calH\ \to \ \calH.
\]
\end{definition}
This is a bounded odd symmetric operator. 

%------------------
\subsection{The Fredholmness}\label{SS:Fredholmness}

By \cite[Lemma~2.6]{Bunke00}, if $D$ satisfies Assumption~\ref{A:gap} and  $f\in C^\infty_g(M,k)^\tau$  is invertible at infinity then  $T_f$ is Fredholm.  The goal of this section is to compute $\ind_\tau T_f$. 

%-----
\begin{lemma}
$\ind_\tau T_f$ depends only on the class $[f]\in C_g(M,k)^\tau/C_0(M,k)^\tau$.
\end{lemma}
\begin{proof}
Let $f_1\in [f]$. Then $f_1-f\in C_0(M,k)^\tau$. 

Let $W^{1,2}(M,E\otimes\CC^k)$ denote the domain of the closure of operator $D$, viewed as a Hilbert space. Then $D:W^{1,2}(M,E\otimes\CC^k)\to L^2(M,E\otimes\CC^k)$ is a bounded operator and $\calH\subset W^{1,2}(M,E\otimes\CC^k)$.
The lemma follows now from Corollary~\ref{C:compact perturbation}.
\end{proof}

%------------------
\newcommand{\calC}{\mathcal{C}}
\subsection{The Callias-type operator}\label{SS:Callias for Boutet}
Let $f\in C_g(M,k)^\tau$ be invertible at infinity. Consider the ungraded odd symmetric operator 
\[
    \calC\ := \ D\ + \  i\,M_f:\, \Gamma(M,E)\ \to \Gamma(M,E).
\]

\begin{lemma}\label{Mf Callias}
The potential $M_f$ satisfies the ungraded admissibility condition of Definition~\ref{D:ungradedCallias}. Hence, $\calC$ is a  Callias-type operator. In particular, it is Fredholm.     
\end{lemma}
\begin{proof}
It follows from the Leibniz rule \eqref{E:Leibniz} that the commutator
\[
    [D,M_f]\ = \ c(df). 
\]
Thus, condition (i) of Definition~\ref{D:ungradedCallias} holds.

Since $F$ is invertible at infinity and $df$ vanishes at infinity, there exists a constant $c$ and a compact set $K\subset M$ such that $f^2(x)>2c$ and $\|df\|<c$. This implies condition (ii) of Definition~\ref{D:ungradedCallias}.
\end{proof}

The main result of this section is the following analog of Proposition~2.6 of \cite{Bunke00}. 

\begin{theorem}\label{T:Toepitz}
$\displaystyle \ind_\tau \calC\ = \ \ind_\tau T_f$.
\end{theorem}
\begin{proof}
Recall that $P:L^2(M,E)\to \calH$ denotes the orthogonal projection. Set $Q:= 1-P$. The image of $Q$ is the orthogonal complement $\calH^\perp$ of $\calH$. 
With respect to decomposition $L^2(M,E\otimes\CC^k= \calH\oplus \calH^\perp$ we have
\[
    \calC\ = \ \begin{pmatrix}
        P\calC P& P\calC Q\\ Q\calC P& Q\calC Q
    \end{pmatrix}
\]

Recall that we denote by $W^{1,2}(M,E\otimes\CC^k)$ the domain of the closure of $D$.
By Lemma~2.4 of \cite{Bunke00} the operators $PM_fQ$ and $QM_fP$ are compact as operators from $W^{1,2}(M,E\otimes\CC^k)$ to $L^2(M,E\otimes\CC^k)$. All these operators are odd symmetric since $P$ is odd symmetric. Hence, by Corollary~\ref{C:compact perturbation}, 
\[
    \ind_\tau \calC\ = \ \ind_\tau\begin{pmatrix}
        P\calC P& 0\\ 0& Q\calC Q 
    \end{pmatrix}\ = \ \ind_\tau P\calC P\ + \ \ind_\tau Q\calC Q.
\]
Consider the family $\calC_t:= D+itM_f$ ($t\in [0,1]$. For all $t>0$, the operator $\calC_t$ is of Callias type and, hence, Fredholm. So $Q\calC_tQ$ is also Fredholm. The operator $Q\calC_0Q=D\big|_{(\ker D)^\perp}$ is invertible. Thus, we have a continuous family $Q\calC_t Q$ of operators connecting $Q\calC Q$ with an invertible operator. It follows that $\ind_\tau Q\calC Q=0$.  

Finally, $P\calC P= T_f$. Hence, 
\[
    \ind_\tau \calC\ = \ \ind_\tau P\calC P \ = \ \ind_\tau T_f.
\]
\end{proof}

\nocite{BrShi21odd}

\end{document}